\documentclass{siamart0516}
\usepackage{amsmath}
\usepackage{amsfonts}
\usepackage{accents}
\usepackage[margin=1in, letterpaper]{geometry}
\usepackage{graphicx}
\usepackage{bmpsize}
\usepackage{subcaption}
\usepackage{float}
\usepackage{epsfig}
\usepackage{multicol}
\usepackage{cite}
\usepackage{adjustbox,lipsum}

\newcommand{\vertiii}[1]{{\left\vert\kern-0.25ex\left\vert\kern-0.25ex\left\vert #1\right\vert\kern-0.25ex\right\vert\kern-0.25ex\right\vert}}
\allowdisplaybreaks
\begin{document}

\title{An Artificial Compressibility Ensemble Timestepping Algorithm for Flow Problems}
\author{J. A. Fiordilino\thanks{The authors' research were partially supported by NSF grants DMS 1522267 and CBET 1609120.  The first author is also supported by the DoD SMART Scholarship.} \and M. McLaughlin\footnotemark[1]}
\date{Updated: 9/5/17}
\maketitle

\begin{abstract}
	Ensemble calculations are essential for systems with uncertain data but require substantial increase in computational resources.  This increase severely limits ensemble size.  To reach beyond current limits, we present a first-order artificial compressibility ensemble algorithm.  This algorithm effectively decouples the velocity and pressure solve via artificial compression, thereby reducing computational complexity and execution time.  Further reductions in storage and computation time are achieved via a splitting of the convective term.  Nonlinear energy stability and first-order convergence of the method are proven under a CFL-type condition involving fluctuations of the velocity.  Numerical tests are provided which confirm the theoretical analyses and illustrate the value of ensemble calculations.
\end{abstract}
\section{Introduction}
The data in physical applications, initial conditions, forcings, and parameters, are never known exactly due to fundamental uncertainty in measurement devices.  The growth of this uncertainty can seriously degrade solution quality.  Ensemble calculations improve solution quality; in particular, the ensemble average is the most likely solution and its variance provides an estimate of prediction reliability.  Typically, computing a solution ensemble involves either J sequential fine mesh runs or J parallel coarse mesh runs of a given code subject to perturbed data.  This leads to the fundamental question: \textbf{Can we increase ensemble size without decreasing mesh density (and vice versa) on a fully utilized computer system?}

Recent breakthroughs in ensemble timestepping algorithms \cite{Fiordilino,Fiordilino3,Gunzburger,Gunzburger2,Jiang,Jiang2,Jiang3,Jiang4,Moheb,Takhirov} reduce memory requirements and computational costs for ensemble simulations.  The same general pattern is followed in each of these works: decomposition of parameters and/or convective velocity into ensemble mean and fluctuating components followed by an IMEX discretization.  The resulting linear systems share the same coefficient matrix, reducing storage and computation time.  Although these works represent a significant advance, there is a need for more efficient algorithms due to ensemble size and resolution demands.  New methodologies must be applied to reach further.  Moreover, we are interested in algorithms with efficiency gains even for an ensemble size of one.

One possible entry point is the saddle point structure.  Operator splitting \cite{Glowinski,Massarotti,Yanenko}, artificial compressibility \cite{Chorin,Decaria,Guermond,Rong,Shen,Shen2,Temam}, and projection methods \cite{Guermond2,Prohl}, among others, address this.  Artificial compressibility, in particular, decouples the velocity and pressure solves, decreasing storage and complexity and increasing speed of computation.

The algorithm presented herein combines two effective tactics for reducing storage requirements and computation time: decoupling velocity, pressure, and temperature solves and keeping the coefficient matrix, at each timestep, constant for each ensemble member.  A CFL-type condition is introduced, which causes breakdown near and into turbulent flow regimes.  A turbulence model should be incorporated into the algorithm, in this event, and is under study.  Consequently, the focus of this paper is on laminar flow.

Consider natural convection within an enclosed cavity with zero wall thickness, e.g. see Figure \ref{figure=domain}.  Let $\Omega \subset \mathbb{R}^d$ (d=2,3) be a convex polyhedral domain with boundary $\partial \Omega$.  The boundary is partitioned such that $\partial \Omega = \overline{\Gamma_{1}}  \cup \overline{\Gamma_{2}}$ with $\Gamma_{1} \cap \Gamma_{2} =\emptyset$, $|\Gamma_{1}| > 0$, $\Gamma_{1} = \Gamma_{H} \cup \Gamma_{N}$, and dist$(\overline{\Gamma_{H}},\overline{\Gamma_{N}})>0$; that is, the boundary is decomposed into a Neumann and Dirichlet part and the Dirichlet boundary is further decomposed into positively separated homogeneous and non-homogeneous parts.  Given $u(x,0;\omega_{j}) = u^{0}(x;\omega_{j})$ and $T(x,0;\omega_{j}) = T^{0}(x;\omega_{j})$ for $j = 1,2, ... , J$, let $u(x,t;\omega_{j}):\Omega \times (0,t^{\ast}] \rightarrow \mathbb{R}^{d}$, $p(x,t;\omega_{j}):\Omega \times (0,t^{\ast}] \rightarrow \mathbb{R}$, and $T(x,t;\omega_{j}):\Omega \times (0,t^{\ast}] \rightarrow \mathbb{R}$ satisfy
\begin{align}
u_{t} + u \cdot \nabla u -Pr \Delta u + \nabla p = PrRa\xi T + f \; \; in \; \Omega, \label{boussinesqv} \\
\nabla \cdot u = 0 \; \; in \; \Omega,  \label{boussinesqm}\\
T_{t} + u \cdot \nabla T - \Delta T = g \; \; in \; \Omega, \label{boussinesqt} \\
u = 0 \; \; on \; \partial \Omega,  \; \; \;
T = 1 \; \; on \; \Gamma_{N}, \; \; \;
T = 0 \; \; on \; \Gamma_{H}, \; \; \;
n \cdot \nabla T = 0 \; \; on \; \Gamma_{2}, \label{boussinesqbcs}
\end{align}
\noindent where $n$ denotes the usual outward normal, $\xi$ denotes the unit vector in the direction of gravity, $Pr$ is the Prandtl number, and $Ra$ is the Rayleigh number.  Further, $f$ and $g$ are body forces and heat sources, respectively.  

Let $<u>^{n} := \frac{1}{J} \sum_{j=1}^{J} u(x,t^{n};\omega_{j})$ and ${u'}^{n} = u(x,t^{n};\omega_{j}) - <u>^{n}$ be the ensemble average and fluctuation. Suppress the spatial discretization for the moment to present the main idea.  We apply an implicit-explicit (IMEX) time-discretization to the system (\ref{boussinesqv}) - (\ref{boussinesqbcs}) such that the resulting coefficient matrix is independent of the ensemble members.  Moreover, we penalize mass conservation by adding a discretized version of the penalty term $\epsilon p_{t}$.  This leads to the artificial compressibility ensemble (ACE) timestepping method:
\begin{align}
\frac{u^{n+1} - u^{n}}{\Delta t} + <u>^{n} \cdot \nabla u^{n+1} + {u'}^{n} \cdot \nabla u^{n} - Pr \Delta u^{n+1} + \nabla p^{n+1} = PrRa\xi T^{n} + f^{n+1}, \label{d1} \\
\epsilon \frac{p^{n+1} - p^{n}}{\Delta t} + \nabla \cdot u^{n+1} = 0, \label{d1m}\\
\frac{T^{n+1} - T^{n}}{\Delta t} + <u>^{n} \cdot \nabla T^{n+1} + {u'}^{n} \cdot \nabla T^{n} - \Delta T^{n+1} = g^{n+1}. \label{d2}
\end{align}
The treatment of the nonlinear terms, $u \cdot \nabla u$ and $u \cdot \nabla T$, leads to a \textit{shared} coefficient matrix, in the above, independent of the ensemble members.  The nonlinear term is the \textit{source} of ensemble dependence in the coefficient matrix.  In particular, using (\ref{d1m}) in (\ref{d1}) and rearranging, 
the following system must be solved, for each $j$:
\begin{align}\label{matrixvel}
\Big(\frac{1}{\Delta t}I + <u>^{n}\cdot \nabla + Pr \Delta + \frac{\Delta t}{\epsilon}\nabla \nabla \cdot \Big)u^{n+1}  = RHS_{u},\\
p^{n+1} = \frac{\Delta t}{\epsilon} \nabla \cdot u^{n+1} + p^{n},\\
\Big(\frac{1}{\Delta t}I + <u>^{n}\cdot \nabla + \Delta \Big) T^{n+1} = RHS_{T}.\label{matrixtemp}
\end{align}
It is clear that the velocity, pressure, and temperature solves are fully decoupled; the saddle-point problem is replaced with a convection-diffusion problem with grad-div stabilization followed by algebraic pressure update, at each timestep.  After finite element spatial discretization, the matrix associated with $<u>^{n} \cdot \nabla$ is independent of the ensemble member due to using the ensemble average as the convective velocity.  The structure of these systems can be exploited with efficient block solvers for linear systems with multiple right-hand-sides; for example, block LU factorizations \cite{Demmel}, block GMRES \cite{Jbilou}, and block BiCGSTAB \cite{El Guennouni}, among others.

 \begin{figure}
 	\centering
 	\includegraphics[height=2.25in, keepaspectratio]{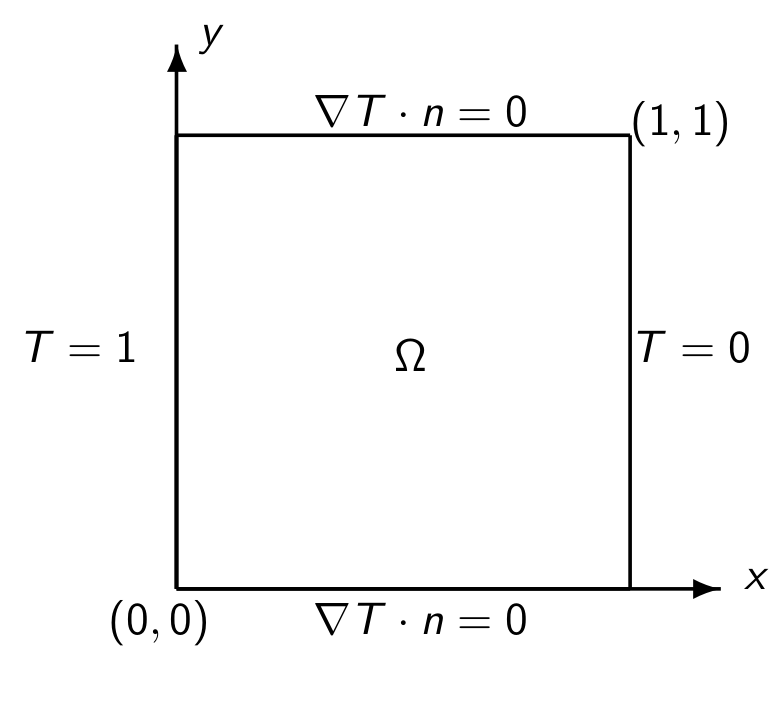}
 	\caption{Domain and boundary conditions for double pane window problem benchmark.}\label{figure=domain}
 \end{figure}

In Section 2, we collect necessary mathematical tools.  In Section 3, we present an algorithm based on (\ref{d1}) - (\ref{d2}) in the context of the finite element method.  Stability and error analysis of the algorithm follow in Section 4.  In particular, under a CFL-type condition, we prove nonlinear energy stability of the proposed algorithm in Theorem \ref{t1} and first-order convergence in Theorem \ref{t2}.   Numerical experiments follow in Section 5 illustrating first-order convergence, speed advantages, and usefulness of ensembles in the context of fluid flow problems.  We end with conclusions in Section 6.
\section{Mathematical Preliminaries}
The $L^{2} (\Omega)$ inner product is $(\cdot , \cdot)$ and the induced norm is $\| \cdot \|$.  Define the Hilbert spaces,
\begin{align*}
X &:= H^{1}_{0}(\Omega)^{d} = \{ v \in H^{1}(\Omega)^d : v = 0 \; on \; \partial \Omega \}, \;
Q := L^{2}_{0}(\Omega) = \{ q \in L^{2}(\Omega) : (1,q) = 0 \}, \\
W &:= H^{1}(\Omega), \; W_{\Gamma_{1}}:= \{ S \in W : S = 0 \; on \; \Gamma_{1} \}, \;
V := \{ v \in X : (q,\nabla \cdot v) = 0 \; \forall \; q \in Q \},
\end{align*}
and $H^{1}(\Omega)$ norm $\| \cdot \|_{1}$.  The dual norm $\| \cdot \|_{-1}$ is understood to correspond to either $X$ or $W_{\Gamma_{1}}$.

We will utilize the fractional order Hilbert space on the non-homogeneous Dirichlet boundary $H^{1/2}(\Gamma_{N})$ with corresponding norm
$$\| R \|_{1/2,\Gamma_{N}} := \Big(\int_{\Gamma_{N}} | R(s) |^{2} ds + \int_{\Gamma_{N}} \int_{\Gamma_{N}} \frac{|R(s)-R(s')|^{2}}{|s-s'|^{d}}dsds'\Big)^{1/2}.$$
Further, let $\tau:\Omega \rightarrow \mathbb{R}$ be an extension of $T\lvert_{\Gamma_{N}} = 1$ into the domain such that $\| \tau\|_{1} \leq C_{tr} \| 1 \|_{1/2,\Gamma_{N}} = C_{tr} \lvert \Gamma_{N}\rvert^{1/2}$ for some $C_{tr}>0$.

\noindent \textbf{Remark:}  The linear conduction profile $\tau(x) = 1 - x_{1}$ for natural convection within a unit square or cube with a pair of differentially heated vertical walls, is such an extension.  It satisfies: $\| \tau\|_{1} \leq \frac{2\sqrt{3}}{3}$; see Lemma 3.2 on p. 1832 of \cite{Colmenares} and references therein for more general domains.

The explicitly skew-symmetric trilinear forms are denoted:
\begin{align*}
b(u,v,w) &= \frac{1}{2} (u \cdot \nabla v, w) - \frac{1}{2} (u \cdot \nabla w, v) \; \; \; \forall u,v,w \in X, \\
b^{\ast}(u,T,S) &= \frac{1}{2} (u \cdot \nabla T, S) - \frac{1}{2} (u \cdot \nabla S, T) \; \; \; \forall u \in X, \; \forall T,S \in W.
\end{align*}
\noindent They enjoy the following continuity properties.
\begin{lemma} \label{l1}
There are constants $C_{1}, C_{2}, C_{3}, C_{4}, C_{5},$ and $C_{6}$ such that for all u,v,w $\in$ X and T,S $\in W$, $b(u,v,w)$ and $b^{\ast}(u,T,S)$ satisfy
\begin{align*}
b(u,v,w) &= (u \cdot \nabla v, w) + \frac{1}{2} ((\nabla \cdot u)v, w), \\
b^{\ast}(u,T,S) &= (u \cdot \nabla T, S) + \frac{1}{2} ((\nabla \cdot u)T, S), \\
b(u,v,w) &\leq C_{1} \| \nabla u \| \| \nabla v \| \| \nabla w \|, \\
b(u,v,w) &\leq C_{2} \sqrt{\| u \| \| \nabla u \|} \| \nabla v \| \| \nabla w \|, \\
b^{\ast}(u,T,S) &\leq C_{3} \| \nabla u \| \| \nabla T \| \| \nabla S \|,\\
b^{\ast}(u,T,S) &\leq C_{4} \sqrt{\| u \| \| \nabla u \|} \| \nabla T \| \| \nabla S \|,\\
b(u,v,w) &\leq C_{5} \| \nabla u \| \| \nabla v \| \sqrt{\| w \| \| \nabla w \|}, \\
b^{\ast}(u,T,S) &\leq C_{6} \| \nabla u \| \| \nabla T \| \sqrt{\| S \| \| \nabla S \|}.
\end{align*}
\begin{proof}
See Lemma 2.1 on p. 12 of \cite{Temam}.
\end{proof}
\end{lemma}
The weak formulation of system (\ref{boussinesqv}) - (\ref{boussinesqbcs}) is:
Find $u:[0,t^{\ast}] \rightarrow X$, $p:[0,t^{\ast}] \rightarrow Q$, $T:[0,t^{\ast}] \rightarrow W$ for a.e. $t \in (0,t^{\ast}]$ satisfying for $j = 1,...,J$:
\begin{align}
({u}_{t},v) + b(u,u,v) + Pr(\nabla u,\nabla v) - (p, \nabla \cdot v) &= PrRa(\xi T,v) + (f,v) \; \; \forall v \in X, \\
(\nabla \cdot u,q) &= 0 \; \; \forall q \in Q, \\
({T}_{t},S) + b^{\ast}(u,T,S) + (\nabla T,\nabla S) &= (g,S) \; \; \forall S \in W_{\Gamma_{1}}.
\end{align}
\subsection{Finite Element Preliminaries}
Consider a quasi-uniform mesh $\Omega_{h} = \{K\}$ of $\Omega$ with maximum triangle diameter length $h$.  Let $X_{h} \subset X$, $Q_{h} \subset Q$, $\hat{W_{h}} = (W_{h},W_{\Gamma_{1},h}) \subset (W,W_{\Gamma_{1}}) = \hat{W}$ be conforming finite element spaces consisting of continuous piecewise polynomials of degrees \textit{j}, \textit{l}, and \textit{j}, respectively.  Moreover, assume they satisfy the following approximation properties $\forall 1 \leq j,l \leq k,m$:
\begin{align}
\inf_{v_{h} \in X_{h}} \Big\{ \| u - v_{h} \| + h\| \nabla (u - v_{h}) \| \Big\} &\leq Ch^{k+1} \lvert u \rvert_{k+1}, \label{a1}\\
\inf_{q_{h} \in Q_{h}}  \| p - q_{h} \| &\leq Ch^{m} \lvert p \rvert_{m}, \label{a2}\\
\inf_{S_{h} \in \hat{W_{h}}}  \Big\{ \| T - S_{h} \| + h\| \nabla (T - S_{h}) \| \Big\} &\leq Ch^{k+1} \lvert T \rvert_{k+1}, \label{a3}
\end{align}
for all $u \in X \cap H^{k+1}(\Omega)^{d}$, $p \in Q \cap H^{m}(\Omega)$, and $T \in \hat{W} \cap H^{k+1}(\Omega)$.  Furthermore, we consider those spaces for which the discrete inf-sup condition is satisfied,
\begin{equation} \label{infsup} 
\inf_{q_{h} \in Q_{h}} \sup_{v_{h} \in X_{h}} \frac{(q_{h}, \nabla \cdot v_{h})}{\| q_{h} \| \| \nabla v_{h} \|} \geq \beta > 0,
\end{equation}
\noindent where $\beta$ is independent of $h$.  Examples include the MINI-element and Taylor-Hood family of elements \cite{John}.  
\\ \indent The Stokes projection will be vital in the upcoming error analysis.  Let $I^{Stokes}_{h}:V \times Q \rightarrow X_{h} \times Q_{h}$ via $I^{Stokes}_{h}(u,p) = (U,P)$ satisfy the following discrete Stokes problem:
\begin{align*}
Pr(\nabla (U - u), \nabla v_{h}) - (P-p, \nabla \cdot v_{h}) = 0 \; \forall \; v_{h} \in X_{h}, \\
(\nabla \cdot (U - u),q_{h}) = 0 \; \forall \; q_{h} \in Q_{h}.
\end{align*}
The following result holds.
\begin{lemma} \label{l2}
Assume the approximation properties \ref{a1}-\ref{a2} and associated regularity hold.  Then, there exists $C>0$ such that
\begin{align*}
h^{-1}\| u - U \| + \| \nabla (u - U) \| + \| p - P \| &\leq C(\beta,Pr,\Omega) \Big\{\inf_{v_{h} \in X_{h}} \| \nabla (u - v_{h}) \| + \inf_{q_{h} \in Q_{h}}  \| p - q_{h} \|\Big\}.
\end{align*}
\end{lemma}
\begin{proof}
	Follows from Theorem 13 on p. 62 of \cite{Layton} and the Aubin-Nitsche trick.
\end{proof}
We will also assume that the finite element spaces satisfy the standard inverse inequality \cite{Ern}:
$$\| \nabla \chi_{1,2} \| \leq C_{inv,1,2} h^{-1} \| \chi_{1,2} \| \; \; \; \forall \chi_{1} \in X_{h}, \; \forall \chi_{2} \in W_{\Gamma_{1},h},$$ 
\noindent where $C_{inv,1,2}$ depend on the minimum angle $\alpha_{min}$ in the triangulation.  A discrete Gronwall inequality will play a  role in the upcoming analysis.
\begin{lemma} \label{l3}
(Discrete Gronwall Lemma). Let $\Delta t$, H, $a_{n}$, $b_{n}$, $c_{n}$, and $d_{n}$ be finite nonnegative numbers for n $\geq$ 0 such that for N $\geq$ 1
\begin{align*}
a_{N} + \Delta t \sum^{N}_{0}b_{n} &\leq \Delta t \sum^{N-1}_{0} d_{n}a_{n} + \Delta t \sum^{N}_{0} c_{n} + H,
\end{align*}
then for all  $\Delta t > 0$ and N $\geq$ 1
\begin{align*}
a_{N} + \Delta t \sum^{N}_{0}b_{n} &\leq exp\big(\Delta t \sum^{N-1}_{0} d_{n}\big)\big(\Delta t \sum^{N}_{0} c_{n} + H\big).
\end{align*}
\end{lemma}
\begin{proof}
	See Lemma 5.1 on p. 369 of \cite{Heywood}.
\end{proof}
Lastly, the discrete time analysis will utilize the following norms $\forall \; -1 \leq k < \infty$:
\begin{align*}
\vertiii{v}_{\infty,k} &:= \max_{1\leq n \leq N} \| v^{n} \|_{k}, \;
\vertiii{v}_{p,k} := \big(\Delta t \sum^{N}_{n = 1} \| v^{n} \|^{p}_{k}\big)^{1/p}.
\end{align*}
\section{Numerical Scheme}
Denote the fully discrete solutions by $u^{n}_{h}$, $p^{n}_{h}$, and $T^{n}_{h}$ at time levels $t^{n} = n\Delta t$, $n = 0,1,...,N$, and $t^{\ast}=N\Delta t$.  For every $n = 0,1,...,N-1$, the fully discrete approximation of (\ref{boussinesqv}) - (\ref{boussinesqbcs}) is:
\\ \underline{Algorithm: ACE}
\\ \textbf{Step one:}  Given $(u^{n}_{h}, p^{n}_{h}, T^{n}_{h})$ $\in (X_{h},Q_{h},W_{h})$, find $(u^{n+1}_{h},T^{n+1}_{h})$ $\in (X_{h},W_{h})$ satisfying
\begin{multline}\label{scheme:one:velocity}
(\frac{u^{n+1}_{h} - u^{n}_{h}}{\Delta t},v_{h}) + b(<u_{h}>^{n},u^{n+1}_{h},v_{h}) + b({u'}^{n}_{h},u^{n}_{h},v_{h}) \\ + Pr(\nabla u^{n+1}_{h},\nabla v_{h}) + \frac{\Delta t}{\epsilon} (\nabla \cdot u^{n+1}_{h}, \nabla \cdot v_{h}) - (p^{n}_{h}, \nabla \cdot v_{h}) =  PrRa(\xi T^{n}_{h},v_{h}) + (f^{n+1},v_{h}) \; \; \forall v_{h} \in X_{h},
\end{multline}
\begin{multline}\label{scheme:one:temperature}
(\frac{T^{n+1}_{h} - T^{n}_{h}}{\Delta t},S_{h}) + b^{\ast}(<u_{h}>^{n},T^{n+1}_{h},S_{h}) + b^{\ast}({u'}^{n}_{h},T^{n}_{h},S_{h}) \\ + (\nabla T^{n+1}_{h},\nabla S_{h})  = (g^{n+1},S_{h}) \; \; \forall S_{h} \in W_{\Gamma_{1},h}.
\end{multline}
\\ \textbf{Step two:} Given $p^{n}_{h} \in Q_{h}$, find $p^{n+1}_{h} \in Q_{h}$ satisfying
\begin{equation} \label{scheme:one:pressure}
p^{n+1}_{h} = p^{n}_{h} - \frac{\Delta t}{\epsilon} \nabla \cdot u^{n+1}.
\end{equation}
\textbf{Remark:} This is a consistent first-order approximation provided $\epsilon = \mathcal{O}(\Delta t^{l+1})$ for $l\geq 0$.  However, the condition number of the resulting system grows without bound as $\Delta t \rightarrow 0$ when $l\geq 1$.
\section{Numerical Analysis of the Ensemble Algorithm}\label{NA}
We present stability results for the aforementioned algorithm under the following timestep condition:
\begin{align}
\frac{C_{\dagger} \Delta t}{h} \max_{1 \leq j \leq J} \|\nabla {u'}^{n}_{h}\|^{2} \leq 1,\label{c1}
\end{align}
\noindent where $C_{\dagger}\equiv C_{\dagger}(\Omega, \alpha_{min},Pr)$.  In the laminar flow regime, condition (\ref{c1}) performs better than conditions appearing in typical explicit methods, where $\|\nabla u_{h}\|$ is present, since $\|\nabla {u'}_{h}\|$ is smaller.

\indent For the artificial compressibility parameter, we prescribe the following $\mathcal{O}(\Delta t)$ relationship, for clarity:
\begin{align}\label{c2}
\epsilon = \gamma^{-1}\Delta t,
\end{align}
\noindent where $\gamma > 0$ is an arbitrary parameter.  Consequently, we have $\frac{\Delta t}{\epsilon} (\nabla \cdot u^{n+1}_{h}, \nabla \cdot v_{h}) = \gamma (\nabla \cdot u^{n+1}_{h}, \nabla \cdot v_{h})$ in equation (\ref{scheme:one:velocity}).  Evidently, the ACE algorithm introduces grad-div stabilization, which is known to have a positive impact on solution quality.  Proper selection of the grad-div parameter $\gamma$ can vary wildly; see e.g. \cite{Jenkins} and references therein.  Further, modest to large values of $\gamma$ are known to dramatically slow down iterative solvers.  Consequently, appropriate choice of $\epsilon$ will vary with application and should be chosen with care.

The remainder of Section \ref{NA} is as follows.  Under condition (\ref{c1}), ACE (\ref{scheme:one:velocity}) - (\ref{scheme:one:temperature}) is proven to be convergent with first-order accuracy in Theorem \ref{t2}.  Nonlinear, energy, stability of the velocity, temperature, and pressure approximations are proven in Theorem \ref{t1}.  Two stability theorems (Theorems \ref{t1b} and \ref{t1c}) are then stated which treat special cases where improvements can be made.
\subsection{Stability Analysis}
\begin{theorem} \label{t1}
Suppose $f \in L^{2}(0,t^{\ast};H^{-1}(\Omega)^{d})$, $g \in L^{2}(0,t^{\ast};H^{-1}(\Omega))$.  If the scheme (\ref{scheme:one:velocity}) - (\ref{scheme:one:temperature}) satisfies condition (\ref{c1}), then
\begin{multline}
\frac{1}{2}\|T^{N}_{h}\|^{2} + \|u^{N}_{h}\|^{2} + \epsilon \|p^{N}_{h}\|^{2} + \frac{1}{2}\sum_{n = 0}^{N-1} \Big(\|T^{n+1}_{h} - T^{n}_{h}\|^{2} + \|u^{n+1}_{h} - u^{n}_{h}\|^{2} + 2\epsilon \|p^{n+1}_{h} - p^{n}_{h}\|^{2}\Big)
\\ + \frac{1}{4} \vertiii{\nabla T_{h}}^{2}_{2,0} + \frac{Pr}{2} \vertiii{\nabla u_{h}}^{2}_{2,0} \leq exp( C_{\#} t^{\ast}) \Big\{\frac{3}{Pr}\vertiii{f}^{2}_{2,-1} + 4\vertiii{g}^{2}_{2,-1} + 4 C_{I}^{2} C{tr}^{2} |\Gamma_{N}| t^{\ast} 
\\ + 3PrRa^{2}C_{PF,1}^2C_{I}^{2}C{tr}^{2} |\Gamma_{N}| t^{\ast} + 2\|T^{0}_{h}\|^{2} + \|u^{0}_{h}\|^{2} + \epsilon \|p^{0}_{h}\|^{2}\Big\} + C^{2}_{I}C^{2}_{tr}|\Gamma_{N}|\big(1 + \frac{t^{\ast}}{2} + 2exp( C_{\#} t^{\ast}) \big).
\end{multline}
\end{theorem}
\begin{proof}
Let $T^{n}_{h} = \theta^{n}_{h} + I_{h}\tau$, where $I_{h}\tau \in W_{h}$ is an interpolant of $\tau$ satisfying $\|I_{h}\tau\|_{1} \leq C_{I} \|\tau\|_{1}$.  We will need the following variational form of equation (\ref{scheme:one:pressure}),
\begin{equation} \label{scheme:one:pressure2}
\epsilon (\frac{p^{n+1}_{h} -p^{n}_{h}}{\Delta t},q_{h}) + (\nabla \cdot u^{n+1}_{h},q_{h}) = 0 \; \; \forall q_{h} \in Q_{h}.
\end{equation}
Use equation (\ref{scheme:one:pressure}) in equation (\ref{scheme:one:velocity}) and add equations (\ref{scheme:one:velocity}), (\ref{scheme:one:temperature}), and (\ref{scheme:one:pressure2}).  Let $(v_{h},q_{h},S_{h})$ \\ $= (2\Delta t u^{n+1}_{h},2\Delta t p^{n+1}_{h},2\Delta t \theta^{n+1}_{h}) \in (V_{h},Q_{h},W_{\Gamma_{1},h})$ and use the polarization identity.  Rearranging yields
\begin{multline}\label{stability:thicku}
\Big\{\|\theta^{n+1}_{h}\|^{2} - \|\theta^{n}_{h}\|  + \|\theta^{n+1}_{h} - \theta^{n}_{h}\|^{2}\Big\} + \Big\{\|u^{n+1}_{h}\|^{2} - \|u^{n}_{h}\|  + \|u^{n+1}_{h} - u^{n}_{h}\|^{2}\Big\}
\\ + \epsilon \Big\{\|p^{n+1}_{h}\|^{2} - \|p^{n}_{h}\|  + \|p^{n+1}_{h} - p^{n}_{h}\|^{2}\Big\} + 2\Delta t \|\nabla \theta^{n+1}_{h}\|^{2} + 2 Pr \Delta t  \|\nabla u^{n+1}_{h}\|^{2} = - 2\Delta t (\nabla I_{h}\tau,\nabla \theta^{n+1}_{h})
\\ - 2\Delta t b^{\ast}({u'}^{n}_{h},\theta^{n}_{h},\theta^{n+1}_{h}) - 2\Delta t b^{\ast}(u^{n}_{h},I_{h}\tau,\theta^{n+1}_{h}) - 2\Delta t b({u'}^{n}_{h},u^{n}_{h},u^{n+1}_{h}) + 2PrRa \Delta t (\xi (\theta^{n}_{h} + I_{h}\tau), u^{n+1}_{h})
\\ + 2\Delta t (f^{n+1},u^{n+1}_{h}) + 2\Delta t (g^{n+1},\theta^{n+1}_{h}).
\end{multline}
Consider $-2\Delta t (\nabla I_{h}\tau,\nabla \theta^{n+1}_{h})$ and $2\Delta t PrRa(\xi I_{h}\tau, u^{n+1}_{h})$.  Use the Cauchy-Schwarz-Young inequality and interpolant estimates on both as well as Poincar\'{e}-Friedrichs on the second,
\begin{align}
-2\Delta t (\nabla I_{h}\tau,\nabla \theta^{n+1}_{h}) &\leq {4\Delta t} \|I_{h}\tau\|^{2}_{1} + \frac{\Delta t}{4} \|\nabla \theta^{n+1}_{h}\|^{2} \leq {4 C_{I}^{2} \Delta t} \|\tau\|^{2}_{1} + \frac{\Delta t}{4} \|\nabla \theta^{n+1}_{h}\|^{2}\label{stability:thick:estinttau}
\\ &\leq {4 C_{I}^{2} C{tr}^{2} |\Gamma_{N}| \Delta t} + \frac{\Delta t}{4} \|\nabla \theta^{n+1}_{h}\|^{2}, \notag
\\ 2\Delta t PrRa(\xi I_{h}\tau, u^{n+1}_{h}) &\leq {3\Delta t PrRa^{2}C_{PF,1}^2C_{I}^{2}C{tr}^{2} |\Gamma_{N}|} + \frac{Pr\Delta t}{3} \| \nabla u^{n+1}_{h} \|^{2}. \label{stability:thick:esttau}
\end{align}
Use the dual norm estimate and Young's inequality on both $2\Delta t (g^{n+1},\theta^{n+1}_{h})$ and $2\Delta t (f^{n+1},u^{n+1}_{h})$.  Cauchy-Schwarz-Young and Poincar\'{e}-Friedrichs inequalities on $2\Delta t PrRa(\xi \theta^{n}_{h}, u^{n+1}_{h})$ yield
\begin{align}
2\Delta t (g^{n+1},\theta^{n+1}_{h}) &\leq {4\Delta t} \|g^{n+1}\|^{2}_{-1} + \frac{\Delta t}{4} \|\nabla \theta^{n+1}_{h}\|^{2}, \label{stability:thick:estg}
\\ 2\Delta t PrRa(\xi \theta^{n}_{h}, u^{n+1}_{h}) &\leq {3\Delta t PrRa^{2}C_{PF,1}^2} \|\theta^{n}_{h}\|^{2} + \frac{Pr\Delta t}{3} \| \nabla u^{n+1}_{h} \|^{2}, \label{stability:thick:estT}
\\ 2\Delta t (f^{n+1},u^{n+1}_{h}) &\leq \frac{3\Delta t}{Pr} \|f^{n+1} \|^{2}_{-1} + \frac{Pr\Delta t}{3} \| \nabla u^{n+1}_{h} \|^{2}. \label{stability:thick:estf}
\end{align}
Consider $-2\Delta t b^{\ast}({u'}^{n}_{h},\theta^{n}_{h},\theta^{n+1}_{h})$ and $2\Delta t b({u'}^{n}_{h},u^{n}_{h},u^{n+1}_{h})$.  Use skew-symmetry, Lemma \ref{l1}, the inverse inequality, and the Cauchy-Schwarz-Young inequality.  Then,
\begin{align}
-2\Delta t b^{\ast}({u'}^{n}_{h},\theta^{n}_{h},\theta^{n+1}_{h}) &= 2\Delta t b^{\ast}({u'}^{n}_{h},\theta^{n+1}_{h},\theta^{n+1}_{h} - \theta^{n}_{h})  \label{stability:thick:estbstar}\\
&\leq 2\Delta t C_{6} \|\nabla {u'}^{n}_{h}\| \|\nabla \theta^{n+1}_{h}\| \sqrt{\|\theta^{n+1}_{h} - \theta^{n}_{h}\| \| \nabla \theta^{n+1}_{h} - \theta^{n}_{h}\|} \notag
\\ &\leq  \frac{2\Delta t C_{6} C^{1/2}_{inv,2}}{h^{1/2}} \|\nabla {u'}^{n}_{h}\| \|\nabla \theta^{n+1}_{h}\| \|\theta^{n+1}_{h} - \theta^{n}_{h}\| \notag
\\ &\leq \frac{2\Delta t^2 C_{6}^2 C_{inv,2}}{h} \|\nabla {u'}^{n}_{h}\|^{2} \|\nabla \theta^{n+1}_{h}\|^{2} + \frac{1}{2} \|\theta^{n+1}_{h} - \theta^{n}_{h}\|^{2}, \notag
\\ -2\Delta t b({u'}^{n}_{h},u^{n}_{h},u^{n+1}_{h}) &\leq \frac{2 \Delta t^{2} C_{5}^{2} C_{inv,1}}{h} \|\nabla {u'}^{n}_{h}\|^{2} \|\nabla u^{n+1}_{h}\|^{2} + \frac{1}{2} \|u^{n+1}_{h} - u^{n}_{h}\|^{2}.\label{stability:thick:estb}
\end{align}
Use the Cauchy-Schwarz-Young, Poincar\'{e}-Friedrichs inequalities and interpolant estimates on \\$-2\Delta t b^{\ast}(u^{n}_{h},I_{h}\tau,\theta^{n+1}_{h})$,
\begin{align}
-2\Delta t b^{\ast}(u^{n}_{h},I_{h}\tau,\theta^{n+1}_{h}) &\leq \Delta t \| u^{n}_{h} \cdot \nabla I_{h} \tau \| \| \theta^{n+1}_{h} \| + \Delta t \| u^{n}_{h} \cdot \nabla \theta^{n+1}_{h} \| \| I_{h} \tau \| \label{stability:thick:estcouplingtau}
\\ &\leq 2C_{I}^{2} C{tr}^{2} |\Gamma_{N}|(1 + C^{2}_{PF_{2}})\Delta t \| u^{n}_{h} \|^{2} + \frac{\Delta t}{4} \| \nabla \theta^{n+1}_{h} \|^{2}.\notag
\end{align}
Using (\ref{stability:thick:estinttau}) - (\ref{stability:thick:estcouplingtau}) in (\ref{stability:thicku}) leads to
\begin{multline}
\big(\|\theta^{n+1}_{h}\|^{2} - \|\theta^{n}_{h}\|\big) + \big(\|u^{n+1}_{h}\|^{2} - \|u^{n}_{h}\|\big) + \epsilon \big(\|p^{n+1}_{h}\|^{2} - \|p^{n}_{h}\|\big) + \frac{1}{2} \Big(\|\theta^{n+1}_{h} - \theta^{n}_{h}\|^{2} + \|u^{n+1}_{h} - u^{n}_{h}\|^{2}\Big)
\\ + \epsilon \|p^{n+1}_{h} - p^{n}_{h}\|^{2} + \frac{\Delta t}{2} \|\nabla \theta^{n+1}_{h}\|^{2} + \frac{Pr \Delta t}{2}  \|\nabla u^{n+1}_{h}\|^{2} + \frac{\Delta t}{2} \| \nabla \theta^{n+1}_{h} \|^{2} \big\{ 1- \frac{4 \Delta t C_{6}^2 C_{inv,2}}{h} \| \nabla {u'}^{n}_{h} \|^{2} \big\}
\\ + \frac{Pr \Delta t}{2} \| \nabla u^{n+1}_{h} \|^{2} \big\{ 1- \frac{4 \Delta t C_{5}^2 C_{inv,1}}{Pr h} \| \nabla {u'}^{n}_{h} \|^{2} \big\} \leq {3\Delta t PrRa^{2}C_{PF,1}^2} \|\theta^{n}_{h}\|^{2}
\\ + 2C_{I}^{2} C{tr}^{2} |\Gamma_{N}|(1 + C^{2}_{PF_{2}})\Delta t \| u^{n}_{h} \|^{2} + {4 C_{I}^{2} C{tr}^{2} |\Gamma_{N}| \Delta t} + {3\Delta t PrRa^{2}C_{PF,1}^2C_{I}^{2}C{tr}^{2} |\Gamma_{N}|}
\\ + \frac{3\Delta t}{Pr} \| f^{n+1} \|_{-1}^{2} + {4\Delta t} \| g^{n+1} \|_{-1}^{2}.
\end{multline}
\noindent Let $C_{\#} = \max\{3PrRa^{2}C^{2}_{PF,1}, 2C_{I}^{2} C{tr}^{2} |\Gamma_{N}|(1 + C^{2}_{PF_{2}}), \epsilon\}$.  Add $\epsilon \Delta t \|p^{n}_{h}\|^{2}$ to the r.h.s. and take a maximum over constants pertaining to Gronwall terms.  Lastly, using the timestep condition \ref{c1}, and summing from $n = 0$ to $n = N-1$ leads to,
\begin{multline}
\|\theta^{N}_{h}\|^{2} + \|u^{N}_{h}\|^{2} + \epsilon \|p^{N}_{h}\|^{2} + \frac{1}{2}\sum_{n = 0}^{N-1} \Big(\|\theta^{n+1}_{h} - \theta^{n}_{h}\|^{2} + \|u^{n+1}_{h} - u^{n}_{h}\|^{2} + 2\epsilon \|p^{n+1}_{h} - p^{n}_{h}\|^{2}\Big)
\\ + \frac{1}{2} \vertiii{\nabla \theta_{h}}^{2}_{2,0} + \frac{Pr}{2} \vertiii{\nabla u_{h}}^{2}_{2,0} \leq C_{\#} \Delta t \sum_{n = 0}^{N-1} \Big(\|\theta^{n}_{h}\|^{2} + \| u^{n}_{h} \|^{2} + \epsilon\|p^{n}_{h}\|^{2} \Big) + 4 C_{I}^{2} C{tr}^{2} |\Gamma_{N}| t^{\ast}
\\ + 3PrRa^{2}C_{PF,1}^2C_{I}^{2}C{tr}^{2} |\Gamma_{N}| t^{\ast} + \frac{3}{Pr}\vertiii{f}^{2}_{2,-1} + 4\vertiii{g}^{2}_{2,-1}.
\end{multline}
Apply Lemma \ref{l3}.  Then,
\begin{multline}
\|\theta^{N}_{h}\|^{2} + \|u^{N}_{h}\|^{2} + \epsilon \|p^{N}_{h}\|^{2} + \frac{1}{2}\sum_{n = 0}^{N-1} \Big(\|\theta^{n+1}_{h} - \theta^{n}_{h}\|^{2} + \|u^{n+1}_{h} - u^{n}_{h}\|^{2} + 2\epsilon \|p^{n+1}_{h} - p^{n}_{h}\|^{2}\Big)
\\ + \frac{1}{2} \vertiii{\nabla \theta_{h}}^{2}_{2,0} + \frac{Pr}{2} \vertiii{\nabla u_{h}}^{2}_{2,0} \leq exp( C_{\#} t^{\ast}) \Big\{\frac{3}{Pr}\vertiii{f}^{2}_{2,-1} + 4\vertiii{g}^{2}_{2,-1} + 4 C_{I}^{2} C{tr}^{2} |\Gamma_{N}| t^{\ast} 
\\ + 3PrRa^{2}C_{PF,1}^2C_{I}^{2}C{tr}^{2} |\Gamma_{N}| t^{\ast} + \|\theta^{0}_{h}\|^{2} + \|u^{0}_{h}\|^{2} + \epsilon \|p^{0}_{h}\|^{2}\Big\}.
\end{multline}
\noindent The result follows by recalling the identity $T^{n}_{h} = \theta^{n}_{h} - I_{h}\tau$ and applying the triangle inequality.  Thus, numerical approximations of velocity, pressure, and temperature are stable.
\end{proof}
\textbf{Remark:}  The exponential growth factor $\exp(C_{\#}t^{\ast})$ can be improved.  In particular, the following Theorems hold.

\begin{theorem} \label{t1b}
Let $\Omega_{h}$ be a regular mesh and suppose the first
meshline in the finite element mesh is within $\mathcal{O}(Ra^{-1})$ of the heated wall $\Gamma_{N}$.  Moreover, suppose $f \in L^{2}(0,t^{\ast};H^{-1}(\Omega)^{d})$, $g \in L^{2}(0,t^{\ast};H^{-1}(\Omega))$.  If the scheme (\ref{scheme:one:velocity}) - (\ref{scheme:one:temperature}) satisfies
\begin{align}
\max_{1 \leq j \leq J} \max_{K \in \Omega_{h}}\frac{C_{\dagger} \Delta t}{h_{K}}\|\nabla {u'}^{n}_{h}\|^{2}_{L^{2}(K)} \leq 1,\label{c1b}
\end{align}
then there exists $C>0$ such that
\begin{multline}
\frac{1}{2}\|T^{N}_{h}\|^{2} + \|u^{N}_{h}\|^{2} + \epsilon \|p^{N}_{h}\|^{2} + \frac{1}{2}\sum_{n = 0}^{N-1} \Big(\|T^{n+1}_{h} - T^{n}_{h}\|^{2} + \|u^{n+1}_{h} - u^{n}_{h}\|^{2} + 2\epsilon \|p^{n+1}_{h} - p^{n}_{h}\|^{2}\Big)
\\ + \frac{1}{4} \vertiii{\nabla T_{h}}^{2}_{2,0} + \frac{Pr}{8} \vertiii{\nabla u_{h}}^{2}_{2,0} \leq C t^{\ast}.
\end{multline}
\end{theorem}
\begin{proof}
This follows from techniques used herein and in \cite{Fiordilino2}.
\end{proof}

\begin{theorem} \label{t1c}
Suppose the hypotheses of Theorem \ref{t1} hold and either $T\equiv 0$ on the entire Dirichlet boundary $\Gamma_{1}$ or wall-thickness is incorporated into the model.  Then, there exists $C>0$ such that the scheme (\ref{scheme:one:velocity}) - (\ref{scheme:one:temperature}) satisfies
\begin{multline}
\|T^{N}_{h}\|^{2} + \|u^{N}_{h}\|^{2} + \epsilon \|p^{N}_{h}\|^{2} + \frac{1}{2}\sum_{n = 0}^{N-1} \Big(\|T^{n+1}_{h} - T^{n}_{h}\|^{2} + \|u^{n+1}_{h} - u^{n}_{h}\|^{2} + 2\epsilon \|p^{n+1}_{h} - p^{n}_{h}\|^{2}\Big)
\\ + \frac{1}{2} \vertiii{\nabla T_{h}}^{2}_{2,0} + \frac{Pr}{2} \vertiii{\nabla u_{h}}^{2}_{2,0} \leq C.
\end{multline}
\end{theorem}
\begin{proof}
This follows from techniques used herein and in \cite{Bernardi,Boland,Fiordilino}.
\end{proof}
\subsection{Error Analysis}
Denote $u^{n}$, $p^{n}$, and $T^{n}$ as the true solutions at time $t^{n} = n\Delta t$.  Assume the solutions satisfy the following regularity assumptions:
\begin{align} 
u &\in L^{\infty}(0,t^{\ast};X \cap H^{k+1}(\Omega)), \;  T, \tau \in L^{\infty}(0,t^{\ast};W \cap H^{k+1}(\Omega)), \notag \\
u_{t}, T_{t} &\in L^{2}(0,t^{\ast};H^{k+1}(\Omega)), \; u_{tt}, T_{tt} \in L^{2}(0,t^{\ast};H^{k+1}(\Omega)), \label{error:regularity} \\
p &\in L^{2}(0,t^{\ast};Q \cap H^{m}(\Omega)), \; p_{t} \in L^{2}(0,t^{\ast};Q(\Omega)). \notag
\end{align}
\noindent \textbf{Remark:}  Regularity of the auxiliary temperature solution $\theta$ follows since $\theta = T - \tau$.  Convergence is proven for $\theta$ first.  The result will follow for the primitive variable $T$ via the triangle inequality and interpolation estimates.  

The errors for the solution variables are denoted
\begin{align*}
e^{n}_{u} &= (u^{n} - U^{n}) - (u^{n}_{h}- U^{n}) = \eta^{n} - \phi^{n}_{h},
\\ e^{n}_{\theta} &= (\theta^{n} - I_{h}\theta^{n}) - (\theta^{n}_{h}-I_{h}\theta^{n}) = \zeta^{n} - \psi^{n}_{h}, 
\\ e^{n}_{p} &= (p^{n} - P^{n}) - (p^{n}_{h}- P^{n}) = \lambda^{n} - \pi^{n}_{h}.  
\end{align*}\vspace*{-\baselineskip}
\begin{definition} (Consistency error).  The consistency errors are denoted
\begin{align*}
\varsigma_{u}(u^{n};v_{h}) &= \big(\frac{u^{n} - u^{n-1}}{\Delta t} - u^{n}_{t}, v_{h}\big) - b(u^{n}-u^{n-1},u^{n},v_{h}) + PrRa(\xi (T^{n}-T^{n-1}),v_{h}),
\\ \varsigma_{p}(p^{n};q_{h}) &= \epsilon \big(\frac{1}{\Delta t} \int^{t^{n}}_{t^{n-1}} p_{t}(s) ds, q_{h}\big),
\\ \varsigma_{T}(T^{n};S_{h}) &= \big(\frac{T^{n} - T^{n-1}}{\Delta t} - T^{n}_{t}, S_{h}\big) - b^{\ast}(u^{n}-u^{n-1},T^{n},S_{h}).
\end{align*}
\end{definition}
\begin{lemma}\label{consistency}
Provided  $u \; and \; T$ satisfy the regularity assumptions \ref{error:regularity}, then $\exists \; C >0$ such that $\forall \; \epsilon,\; r > 0$
\begin{align*}
\lvert \varsigma_{u}(u^{n};v_{h}) \rvert &\leq \frac{C C^{2}_{PF,1} C_{r} \Delta t}{\delta}\| u_{tt}\|^{2}_{L^{2}(t^{n-1},t^{n};L^{2}(\Omega))} + \frac{C_{1}^{2}C_{r}\Delta t}{\delta}\| \nabla u^{n}\|^{2}\| \nabla u_{t}\|^{2}_{L^{2}(t^{n-1},t^{n};L^{2}(\Omega))}
\\ &+ \frac{C^{2}_{PF,1}C_{r}\Delta t}{\delta}\| T_{t}\|^{2}_{L^{2}(t^{n-1},t^{n};L^{2}(\Omega))} + \frac{\delta}{r} \| \nabla v_{h} \|^{2},
\\ \lvert \varsigma_{T}(T^{n};S_{h}) \rvert &\leq \frac{C C^{2}_{PF,2} C_{r} \Delta t}{\delta}\| T_{tt}\|^{2}_{L^{2}(t^{n-1},t^{n};L^{2}(\Omega))} + \frac{C_{3}^{2}C_{r}\Delta t}{\delta}\| \nabla T^{n}\|^{2}\| \nabla u_{t}\|^{2}_{L^{2}(t^{n-1},t^{n};L^{2}(\Omega))} + \frac{\delta}{r} \| \nabla S_{h} \|^{2}.
\end{align*}
\end{lemma}
\begin{proof}
These follow from the Cauchy-Schwarz-Young inequality, Poincar\'{e}-Friedrichs inequality, and Taylor's Theorem with integral remainder.
\end{proof}
\begin{theorem} \label{t2}
For (u,p,T) satisfying (1) - (5), suppose that $(u^{0}_{h},p^{0}_{h},T^{0}_{h}) \in (X_{h},Q_{h},W_{h})$ are approximations of $(u^{0},p^{0},T^{0})$ to within the accuracy of the interpolant.  Further, suppose that condition (\ref{c1}) holds. Then there exists a constant $C>0$ such that
\begin{multline*}
\frac{1}{2}\|e^{N}_{T}\|^{2} + \|e^{N}_{u}\|^{2} + \gamma^{-1} \Delta t \|e^{N}_{p}\|^{2} + \frac{1}{2}\sum^{N-1}_{n=0} \Big\{\|e^{n+1}_{T}-e^{n}_{T}\|^{2} + \|e^{n+1}_{u}-e^{n}_{u}\|^{2} + \gamma^{-1} \Delta t \|e^{n+1}_{p} - e^{n}_{p}\|^{2}\Big\}
\\  + \frac{Pr\Delta t}{4} \|\nabla e^{N}_{u}\|^{2} + \frac{1}{4}\vertiii{\nabla e_{T}}^{2}_{2,0} + \frac{Pr}{2} \vertiii{\nabla e_{u}}^{2}_{2,0} \leq C exp(C_{\star} t^{\ast}) \inf_{\substack{v_{h} \in X_{h}\\q_{h} \in Q_{h}\\S_{h} \in \hat{W}_{h}}} \Big\{ \| (T-S_{h})_{t} \|^{2}_{L^{2}(0,t^{\ast};L^{2}(\Omega))}
\\ + \| (u-v_{h})_{t} \|^{2}_{L^{2}(0,t^{\ast};L^{2}(\Omega))} + \Delta t^{2} \| (p-q_{h})_{t} \|^{2}_{L^{2}(0,t^{\ast};L^{2}(\Omega))} + h\Delta t \| \nabla (T-S_{h})_{t} \|^{2}_{L^{2}(0,t^{\ast};L^{2}(\Omega))}
\\ + h \Delta t \| \nabla (u-v_{h})_{t} \|^{2}_{L^{2}(0,t^{\ast};L^{2}(\Omega))} + \vertiii{T-S_{h}}^{2}_{2,0} + \vertiii{\nabla (T-S_{h})}^{2}_{2,0} + \vertiii{T-S_{h}}_{2,0} \vertiii{\nabla (T-S_{h})}_{2,0}
\\ + \vertiii{u-v_{h}}_{2,0} \vertiii{\nabla (u-v_{h})}_{2,0} + t^{\ast}\big(\|\tau - I_{h}\tau\|^{2} + \|\nabla (\tau - I_{h}\tau)\|^{2}\big) + \Delta t^{2} + h\Delta t
\\ + 2\|e^{0}_{T}\|^{2} + \|e^{0}_{u}\|^{2} + \gamma^{-1} \Delta t \|e^{0}_{p}\|^{2} + \frac{Pr\Delta t}{4}\|\nabla e^{0}_{u}\|^{2}\Big\}
\end{multline*}
\end{theorem}
\begin{proof}
Let $T^{n} = \theta^{n} + \tau$.  The true solutions satisfy for all $n = 0, ... N-1$:
\begin{align}
(\frac{u^{n+1} - u^{n}}{\Delta t},v_{h}) + b(u^{n},u^{n+1},v_{h}) + Pr(\nabla u^{n+1},\nabla v_{h}) - (p^{n+1}, \nabla \cdot v_{h}) \label{error:one:truevelocity}
\\ =  PrRa(\xi (\theta^{n}+\tau),v_{h}) + (f^{n+1},v_{h}) + \varsigma_{u}(u^{n+1};v_{h}) \; \; \forall v_{h} \in X_{h},\notag
\\ \epsilon(\frac{p^{n+1}-p^{n}}{\Delta t}) +(\nabla \cdot u^{n+1},q_{h}) = \varsigma_{p}(p^{n+1};q_{h}) \; \; \forall q_{h} \in Q_{h} \label{error:one:truepressure},
\\ (\frac{\theta^{n+1} - \theta^{n}}{\Delta t},S_{h}) + b^{\ast}(u^{n},\theta^{n+1},S_{h}) + (\nabla \theta^{n+1},\nabla S_{h}) + (\nabla \tau,\nabla S_{h}) \label{error:one:truetemp}
\\ = (g^{n+1},S_{h}) + \varsigma_{T}(\theta^{n+1};S_{h}) \; \; \forall S_{h} \in W_{\Gamma_{1,h}}.\notag
\end{align}
Subtract (\ref{scheme:one:temperature}) from (\ref{error:one:truetemp}), then the error equation for temperature is
\begin{align}
(\frac{e^{n+1}_{\theta} - e^{n}_{\theta}}{\Delta t},S_{h}) + b^{\ast}(u^{n},\theta^{n+1},S_{h}) - b^{\ast}(<u_{h}>^{n},\theta^{n+1}_{h},S_{h}) - b^{\ast}({u'}^{n}_{h},\theta^{n}_{h},S_{h})
\\ + b^{\ast}(u^{n},\tau,S_{h})-b^{\ast}(u^{n}_{h},I_{h}\tau,S_{h}) + (\nabla e^{n+1}_{\theta},\nabla S_{h}) + (\nabla (\tau - I_{h}\tau),\nabla S_{h})  = \varsigma_{T}(\theta^{n+1},S_{h}) \; \; \forall S_{h} \in W_{\Gamma_{1,h}} \notag.
\end{align}
Decomposing the error terms and rearranging gives,
\begin{multline*}
(\frac{\psi^{n+1}_{h} - \psi^{n}_{h}}{\Delta t},S_{h}) + (\nabla \psi^{n+1}_{h},\nabla S_{h}) = (\frac{\zeta^{n+1} - \zeta^{n}}{\Delta t},S_{h}) + (\nabla \zeta^{n+1},\nabla S_{h})
\\ + (\nabla (\tau - I_{h}\tau),\nabla S_{h}) + b^{\ast}(u^{n},\theta^{n+1},S_{h}) - b^{\ast}(u^{n}_{h},\theta^{n+1}_{h},S_{h}) - b^{\ast}({u'}^{n}_{h},\theta^{n}_{h},S_{h})
\\ + b^{\ast}(u^{n},\tau,S_{h})-b^{\ast}(u^{n}_{h},I_{h}\tau,S_{h}) - \varsigma_{T}(\theta^{n+1},S_{h}) \; \; \forall S_{h} \in W_{\Gamma_{1,h}}.
\end{multline*}
Setting $S_{h} = 2\Delta t \psi^{n+1}_{h} \in W_{\Gamma_{1,h}}$ yields
\begin{multline*}
\Big\{\|\psi^{n+1}_{h}\|^{2} - \|\psi^{n}_{h}\|^{2} + \|\psi^{n+1}_{h}-\psi^{n}_{h}\|^{2}\Big\} + 2\Delta t \|\nabla \psi^{n+1}_{h}\|^{2} = (\zeta^{n+1}-\zeta^{n},\psi^{n+1}_{h}) + 2\Delta t (\nabla \zeta^{n+1},\nabla \psi^{n+1}_{h})
\\ + 2\Delta t (\nabla (\tau - I_{h}\tau),\nabla \psi^{n+1}_{h}) + 2\Delta t b^{\ast}(u^{n},\theta^{n+1},\psi^{n+1}_{h}) - 2\Delta t b^{\ast}(u^{n}_{h},\theta^{n+1}_{h},\psi^{n+1}_{h}) - 2\Delta t b^{\ast}({u'}^{n}_{h},\theta^{n}_{h},\psi^{n+1}_{h})
\\ + 2\Delta t b^{\ast}(u^{n+1},\tau,\psi^{n+1}_{h}) - 2\Delta t b^{\ast}(u^{n}_{h},I_{h}\tau,\psi^{n+1}_{h}) - 2\Delta t \varsigma_{T}(\theta^{n+1},\psi^{n+1}_{h}).
\end{multline*}
Add and subtract $2\Delta t b^{\ast}(u^{n},\theta^{n+1}_{h},\psi^{n+1}_{h})$, $2\Delta t b^{\ast}({u'}^{n}_{h},\theta^{n+1}-\theta^{n},\psi^{n+1}_{h})$, and $2\Delta t b^{\ast}(u^{n},\tau-I_{h}\tau,\psi^{n+1}_{h})$ to the r.h.s.  Rearrange and use skew-symmetry, then
\begin{multline}\label{fet1}
\Big\{\|\psi^{n+1}_{h}\|^{2} - \|\psi^{n}_{h}\|^{2} + \|\psi^{n+1}_{h}-\psi^{n}_{h}\|^{2}\Big\} + 2\Delta t \|\nabla \psi^{n+1}_{h}\|^{2} = 2(\zeta^{n+1}-\zeta^{n},\psi^{n+1}_{h}) + 2\Delta t (\nabla \zeta^{n+1},\nabla \psi^{n+1}_{h})
\\ + 2\Delta t (\nabla (\tau - I_{h}\tau),\nabla \psi^{n+1}_{h}) + 2\Delta t b^{\ast}(u^{n},\zeta^{n+1},\psi^{n+1}_{h}) + 2\Delta t b^{\ast}(\eta^{n},\theta^{n+1}_{h},\psi^{n+1}_{h}) - 2\Delta t b^{\ast}(\phi^{n}_{h},\theta^{n+1}_{h},\psi^{n+1}_{h})
\\ - 2\Delta t b^{\ast}({u'}^{n}_{h},\zeta^{n+1}-\zeta^{n},\psi^{n+1}_{h}) - 2\Delta t b^{\ast}({u'}^{n}_{h},\psi^{n}_{h},\psi^{n+1}_{h}) + 2\Delta t b^{\ast}({u'}^{n}_{h},\theta^{n+1}-\theta^{n},\psi^{n+1}_{h})
\\ + 2\Delta t b^{\ast}(u^{n+1}-u^{n},\tau,\psi^{n+1}_{h}) + 2\Delta t b^{\ast}(\eta^{n},I_{h}\tau,\psi^{n+1}_{h}) - 2\Delta t b^{\ast}(\phi^{n}_{h},I_{h}\tau,\psi^{n+1}_{h}) 
\\ + 2\Delta t b^{\ast}(u^{n},\tau-I_{h}\tau,\psi^{n+1}_{h}) - 2\Delta t \varsigma_{T}(\theta^{n+1},\psi^{n+1}_{h}).
\end{multline}
Follow analogously for the velocity error equation.  Subtract (\ref{scheme:one:velocity}) from (\ref{error:one:truevelocity}).  Let $v_{h} = 2\Delta t \phi^{n+1}_{h} \in X_{h}$, add and subtract $b(u^{n},u^{n+1}_{h},\phi^{n+1}_{h})$ and $b({u'}^{n}_{h},u^{n+1}-u^{n},\phi^{n+1}_{h})$, rearrange and use skew-symmetry.  Then,
\begin{multline}\label{feu1}
\Big\{\|\phi^{n+1}_{h}\|^{2} - \|\phi^{n}_{h}\|^{2} + \|\phi^{n+1}_{h}-\phi^{n}_{h}\|^{2}\Big\} + 2 Pr \Delta t \|\nabla \phi^{n+1}_{h}\|^{2} - 2 \Delta t (\pi^{n+1}_{h},\nabla \cdot \phi^{n+1}_{h}) = 2(\eta^{n+1}-\eta^{n},\phi^{n+1}_{h})
\\ - 2PrRa\Delta t(\xi \zeta^{n},\phi^{n+1}_{h}) + 2PrRa\Delta t (\xi \psi^{n}_{h},\phi^{n+1}_{h}) - 2PrRa\Delta t(\xi (\tau - I_{h}\tau),\phi^{n+1}_{h}) + 2 \Delta t b(u^{n},\eta^{n+1},\phi^{n+1}_{h})
\\ + 2 \Delta t b(\eta^{n},u^{n+1}_{h},\phi^{n+1}_{h}) - 2 \Delta t b(\phi^{n}_{h},u^{n+1}_{h},\phi^{n+1}_{h}) - 2 \Delta t b({u'}^{n}_{h},\eta^{n+1}-\eta^{n},\phi^{n+1}_{h}) - 2 \Delta t b({u'}^{n}_{h},\phi^{n}_{h},\phi^{n+1}_{h})
\\ + 2 \Delta t b({u'}^{n}_{h},u^{n+1}-u^{n},\phi^{n+1}_{h}) - 2 \Delta t \varsigma_{u}(u^{n+1},\phi^{n+1}_{h}).
\end{multline}
Similarly, for the pressure equation, subtract (\ref{scheme:one:pressure2}) from (\ref{error:one:truepressure}).  Let $q_{h} = 2\Delta t \pi^{n+1}_{h} \in Q_{h}$ and rearrange, then
\begin{multline}\label{fep1}
\epsilon \Big\{ \|\pi^{n+1}_{h}\|^{2} - \|\pi^{n}_{h}\|^{2} + \|\pi^{n+1}_{h} - \pi^{n}_{h}\|^{2}\Big\} + 2\Delta t(\nabla \cdot \phi^{n+1}_{h},\pi^{n+1}_{h})
\\ = 2\epsilon (\lambda^{n+1}-\lambda^{n},\pi^{n+1}_{h}) - 2 \Delta t \varsigma_{p}(p^{n+1},\pi^{n+1}_{h}).
\end{multline}
We seek to now estimate all terms on the r.h.s. in such a way that we may subsume the terms involving unknown pieces $\psi^{k}_{h}$, $\phi^{k}_{h}$, and $\pi^{k}_{h}$ into the l.h.s.  The following estimates are formed using skew-symmetry, Lemma \ref{l1}, and the Cauchy-Schwarz-Young inequality,
\begin{align}
2\Delta t b^{\ast}(u^{n},\zeta^{n+1},\psi^{n+1}_{h}) &\leq 2C_{6} \Delta t\| \nabla u^{n} \| \| \nabla \psi^{n+1}_{h} \| \sqrt{\| \zeta^{n+1} \| \| \nabla \zeta^{n+1} \|}  
\\ &\leq \frac{4C_{r} C_{6}^{2}\Delta t}{\delta_4} \| \nabla u^{n} \|^{2} \| \| \zeta^{n+1} \| \| \nabla \zeta^{n+1} \| + \frac{\delta_4 \Delta t}{r} \| \nabla \psi^{n+1}_{h} \|^{2},\notag 
\\ 2\Delta t b^{\ast}(\eta^{n},\theta^{n+1}_{h},\psi^{n+1}_{h}) &\leq \frac{4C_{r} C_{4}^{2}}{\delta_5} \| \nabla \theta^{n+1}_{h} \|^{2}\| \eta^{n} \| \| \nabla \eta^{n} \| + \frac{\delta_5}{r} \| \nabla \psi^{n+1}_{h} \|^{2}. \notag 
\end{align}
Applying Lemma \ref{l1}, the Cauchy-Schwarz-Young inequality, Taylor's theorem, and condition \ref{c1}  yields,
\begin{align}
-2\Delta t b^{\ast}({u'}^{n}_{h},\zeta^{n+1}-\zeta^{n},\psi^{n+1}_{h}) &\leq C_{3} \| \nabla {u'}^{n}_{h} \| \| \nabla \psi^{n+1}_{h} \|\| \nabla \zeta^{n+1}-\zeta^{n})\|
\\ &\leq \frac{4C_{r} C_{3}^{2} \Delta t^{2}}{\delta_7} \| \nabla {u'}^{n}_{h} \|^{2}\| \nabla \zeta_{t} \|^{2}_{L^{2}(t^{n},t^{n+1};L^{2}(\Omega))} + \frac{\delta_7\Delta t}{r} \| \nabla \psi^{n+1}_{h} \|^{2}, \notag
\\ &\leq \frac{4C_{r} C_{3}^{2} h \Delta t}{C_{\dagger}\delta_7}\| \nabla \zeta_{t} \|^{2}_{L^{2}(t^{n},t^{n+1};L^{2}(\Omega))} + \frac{\delta_7\Delta t}{r} \| \nabla \psi^{n+1}_{h} \|^{2}, \notag
\\ 2\Delta t b^{\ast}({u'}^{n}_{h},\theta^{n+1}-\theta^{n},\psi^{n+1}_{h}) &\leq C_{3} \| \nabla {u'}^{n}_{h} \| \| \nabla (\theta^{n+1}-\theta^{n}) \| \| \nabla \psi^{n+1}_{h} \| 
\\ &\leq \frac{4C_{r} C_{3}^{2}h \Delta t}{C_{\dagger}\delta_{9}}\| \nabla \theta_{t} \|^{2}_{L^{2}(t^{n},t^{n+1};L^{2}(\Omega))} + \frac{\delta_{9}\Delta t}{r} \| \nabla \psi^{n+1}_{h}\|^{2}.\notag
\end{align}
Apply the triangle inequality, Lemma \ref{l1} and the Cauchy-Schwarz-Young inequality twice.  This yields
\begin{align}
-2\Delta t b^{\ast}(\phi^{n}_{h},\theta^{n+1}_{h},\psi^{n+1}_{h}) &\leq 2C_{4}\Delta t \| \nabla \theta^{n+1}_{h} \| \| \nabla \psi^{n+1}_{h} \| \sqrt{\| \phi^{n}_{h}\| \| \nabla \phi^{n}_{h}\|} 
\\ &\leq 2C_{4} C_{\theta}(j) \Delta t \| \nabla \psi^{n+1}_{h} \| \sqrt{\| \phi^{n}_{h}\| \| \nabla \phi^{n}_{h} \|} \notag
\\ &\leq \delta_{6}\Delta t \| \nabla \psi^{n+1}_{h} \|^2 + \frac{C_{4}^{2} C_{\theta}^{2}\Delta t}{\delta_{6}} \| \phi^{n}_{h}\| \| \nabla \phi^{n}_{h} \| \notag
\\ &\leq \delta_{6}\Delta t \| \nabla \psi^{n+1}_{h} \|^2 + \frac{C_{4}^{2} C_{\theta}^{2}\Delta t}{2\delta_{6}\sigma_{6}} \| \phi^{n}_{h}\|^{2} + \frac{C_{4}^{2} C_{\theta}^{2}\sigma_{6}\Delta t}{2\delta_{6}}\| \nabla \phi^{n}_{h} \|^{2}, \notag
\\ - 2\Delta t b^{\ast}(\phi^{n}_{h},I_{h}\tau,\psi^{n+1}_{h}) &\leq \delta_{12}\Delta t\| \nabla \psi^{n+1}_{h} \|^2 + \frac{C_{4}^{2} C_{I}^{2}C_{tr}^{2} |\Gamma_{N}|\Delta t}{2\delta_{12}\sigma_{12}} \| \phi^{n}_{h}\|^{2} + \frac{C_{4}^{2} C_{I}^{2}C_{tr}^{2} |\Gamma_{N}|\sigma_{12}\Delta t}{2\delta_{12}}\| \nabla \phi^{n}_{h} \|^{2}.
\end{align}
Use Lemma \ref{l1}, the inverse inequality, and the Cauchy-Schwarz-Young inequality yielding
\begin{align}
2 \Delta t b^{\ast}({u'}^{n}_{h},\psi^{n}_{h},\psi^{n+1}_{h}) &\leq \frac{2C_{6} C^{1/2}_{inv,2}\Delta t}{h^{1/2}} \| \nabla {u'}^{n}_{h} \| \| \nabla \psi^{n+1}_{h} \| \| \psi^{n+1}_{h}-\psi^{n}_{h}\|
\\ &\leq \frac{2C_{6}^{2} C_{inv,2} \Delta t^2}{h} \| \nabla {u'}^{n}_{h} \|^{2} \| \nabla \psi^{n+1}_{h} \|^{2} + \frac{1}{2} \| \psi^{n+1}_{h}-\psi^{n}_{h}\|^{2}.\notag
\end{align}
Use the Cauchy-Schwarz-Young inequality on the first term.  Apply Lemma \ref{l1}, interpolant estimates, and Taylor's theorem on the remaining.  Then,
\begin{align}
2\Delta tb^{\ast}(u^{n+1}-u^{n},\tau,\psi^{n+1}_{h}) &\leq C_{3} \| \nabla (u^{n+1}-u^{n}) \| \|\nabla \tau\| \| \nabla \psi^{n+1}_{h} \| 
\\ &\leq \frac{4C_{r}C_{3}^2 C_{tr}^{2} |\Gamma_{N}| \Delta t^{2}}{\delta_{10}} \| \nabla u_{t} \|^{2}_{L^{2}(t^{n},t^{n+1};L^{2}(\Omega))} + \frac{\delta_{10}\Delta t}{r} \|\nabla \psi^{n+1}_{h} \|, \notag
\\ 2\Delta tb^{\ast}(u^{n},\tau - I_{h}\tau,\psi^{n+1}_{h}) &\leq C_{3}\| \nabla u^{n}\| \|\nabla (\tau - I_{h}\tau) \| \| \nabla \psi^{n+1}_{h} \|
\\ &\leq \frac{4C_{r} C_{3}^{2}\Delta t}{\delta_{13}}\| \nabla u^{n} \|^{2} \|\nabla \tau - I_{h}\tau \|^{2} + \frac{\delta_{13}\Delta t}{r} \| \nabla \psi^{n+1}_{h} \|^{2}, \notag
\\ 2\Delta tb^{\ast}(\eta^{n},I_{h}\tau,\psi^{n+1}_{h}) &\leq C_{4}\|\nabla I_{h}\tau \| \| \nabla \psi^{n+1}_{h}\| \sqrt{\| \eta^{n}\| \| \nabla \eta^{n}\|}  
\\ &\leq \frac{4C_{r} C_{4}^{2}C_{I}^{2}C_{tr}^{2}|\Gamma_{N}|\Delta t}{\delta_{11}} \| \eta^{n}\| \| \nabla \eta^{n}\| + \frac{\delta_{11}\Delta t}{r} \| \nabla \psi^{n+1}_{h} \|^{2}, \notag
\\ 2\Delta t(\nabla (\tau - I_{h}\tau),\nabla \psi^{n+1}_{h}) &\leq \frac{4C_{r}\Delta t}{\delta_{3}} \|\nabla (\tau - I_{h}\tau)\|^{2} + \frac{\delta_{3}\Delta t}{r} \| \nabla \psi^{n+1}_{h} \|^{2}.
\end{align}
The Cauchy-Schwarz-Young inequality, Poincar\'{e}-Friedrichs inequality and Taylor's theorem yield
\begin{align}
2(\zeta^{n+1} - \zeta^{n}, \psi^{n+1}_{h}) \leq \frac{4C^{2}_{PF,2} C_{r}}{\delta_1} \| \zeta_{t} \|^{2}_{L^{2}(t^{n},t^{n+1};L^{2}(\Omega))} + \frac{\delta_1 \Delta t}{r} \| \nabla \psi^{n+1}_{h} \|^{2}.
\end{align}
Lastly, use the Cauchy-Schwarz-Young inequality,
\begin{align}
2\Delta t(\nabla \zeta^{n+1},\nabla \psi^{n+1}_{h}) \leq \frac{4C_{r}\Delta t}{\delta_2} \| \nabla \zeta^{n+1} \|^{2} + \frac{\delta_2 \Delta t}{r} \| \nabla \psi^{n+1}_{h} \|^{2}.
\end{align}
Similar estimates follow for the r.h.s. terms in (\ref{feu1}), however, we must treat additional error terms associated with the temperature,
\begin{align}
- 2PrRa\Delta t(\xi \zeta^{n},\phi^{n+1}_{h}) &\leq \frac{4Pr^{2} Ra^{2} C^{2}_{PF,1} C_{r}\Delta t}{\delta_{16}}\| \zeta^{n} \|^{2} + \frac{\delta_{16}\Delta t}{r} \| \nabla \phi^{n+1}_{h} \|^{2},
\\ 2PrRa\Delta t(\xi \psi^{n}_{h},\phi^{n+1}_{h}) &\leq \frac{4Pr^{2} Ra^{2} C^{2}_{PF,1} C_{r}\Delta t}{\delta_{17}}\|\psi^{n}_{h}\|^{2} + \frac{\delta_{17}\Delta t}{r} \| \nabla \phi^{n+1}_{h} \|^{2},
\\ - 2PrRa\Delta t(\xi (\tau - I_{h}\tau),\phi^{n+1}_{h}) &\leq \frac{4Pr^{2} Ra^{2} C^{2}_{PF,1} C_{r}\Delta t}{\delta_{18}}\| \tau - I_{h}\tau \|^{2} + \frac{\delta_{18}\Delta t}{r} \| \nabla \phi^{n+1}_{h} \|^{2}.
\end{align}
Consider equation (\ref{fep1}).  Add and subtract $2\epsilon (\lambda^{n+1} - \lambda^{n},\pi^{n}_{h})$ and $-2\Delta t \upsilon_{p}(p^{n+1},\pi^{n}_{h})$.  Use Taylor's theorem and the Cauchy-Schwarz-Young inequality.  This leads to
\begin{align}
2\epsilon (\lambda^{n+1} - \lambda^{n},\pi^{n+1}_{h}) &= 2\epsilon (\lambda^{n+1} - \lambda^{n},\pi^{n+1}_{h}-\pi^{n}_{h}) + 2\epsilon (\lambda^{n+1} - \lambda^{n},\pi^{n}_{h}) 
\\ &\leq \frac{4\epsilon C_{r}\Delta t^{2}}{\delta_{26}} \| \lambda_{t} \|^{2}_{L^{2}(t^{n},t^{n+1};L^{2}(\Omega))} + \frac{\epsilon \delta_{26}}{r} \| \pi^{n+1}_{h}-\pi^{n}_{h} \|^{2}\notag
\\ &+ \frac{4\epsilon C_{r}\Delta t}{\delta_{27}} \| \lambda_{t} \|^{2}_{L^{2}(t^{n},t^{n+1};L^{2}(\Omega))} + \frac{\epsilon \delta_{27}\Delta t}{r} \| \pi^{n}_{h} \|^{2},\notag
\\ -2\Delta t \varsigma_{p}(p^{n+1},\pi^{n+1}_{h})  &\leq \frac{4\epsilon C_{r}\Delta t^{2}}{\delta_{28}} \| p_{t} \|^{2}_{L^{2}(t^{n},t^{n+1};L^{2}(\Omega))} + \frac{\epsilon \delta_{28}}{r} \| \pi^{n+1}_{h}-\pi^{n}_{h} \|^{2}
\\ &+ \frac{4\epsilon C_{r}\Delta t}{\delta_{29}} \| p_{t} \|^{2}_{L^{2}(t^{n},t^{n+1};L^{2}(\Omega))} + \frac{\epsilon \delta_{29}\Delta t}{r} \| \pi^{n}_{h} \|^{2}.\notag
\end{align}
Add equations (\ref{fet1}) - (\ref{fep1}) together.  Apply the above estimates and Lemma \ref{consistency}.  Let $r = 40$ and choose $\sum^{14}_{i\neq6,12} \delta_{i} = 10$, $\delta_{6} = \delta_{12} = 1/8$, $\sum^{25}_{i\neq21} \delta_{i} = 10$, $\delta_{21} = 1/8$, and $\delta_{26}=\delta_{28}=10$.  Moreover, let $\sigma_{6} = \frac{\delta_{6}}{12C^{2}_{4}C^{2}_{\theta}}$, $\sigma_{12} = \frac{\delta_{12}}{12C^{2}_{4}C^{2}_{I}C^{2}_{tr}|\Gamma_{N}|}$, and $\sigma_{21} = \frac{\delta_{21}}{12C_{2}^{2} C_{u}^{2}}$.  Reorganize, use condition (\ref{c1}), relation (\ref{c2}), and Theorem \ref{t1}.  Take the maximum over all constants associated with $\| \psi^{n}_{h} \|$, $\| \phi^{n}_{h} \|$, and $\| \pi^{n}_{h} \|$ on the r.h.s.  Lastly, take the maximum over all remaining constants on the r.h.s.  Then,
\begin{multline} 
\Big\{\|\psi^{n+1}_{h}\|^{2} - \|\psi^{n}_{h}\|^{2}\Big\} + \Big\{\|\phi^{n+1}_{h}\|^{2} - \|\phi^{n}_{h}\|^{2}\Big\} + \gamma^{-1} \Delta t \Big\{ \|\pi^{n+1}_{h}\|^{2} - \|\pi^{n}_{h}\|^{2}\Big\}
\\ + \frac{1}{2} \Big\{\|\psi^{n+1}_{h}-\psi^{n}_{h}\|^{2} + \|\phi^{n+1}_{h}-\phi^{n}_{h}\|^{2} + \gamma^{-1} \Delta t \|\pi^{n+1}_{h} - \pi^{n}_{h}\|^{2}\Big\} + \frac{\Delta t}{2} \Big\{\|\nabla \psi^{n+1}_{h}\|^{2} + Pr \|\nabla \phi^{n+1}_{h}\|^{2}\Big\}
\\ + \frac{Pr\Delta t}{4} \Big\{\|\nabla \phi^{n+1}_{h}\|^{2} - \|\nabla \phi^{n}_{h}\|^{2} \Big\} \leq C_{\star} \Delta t \Big\{ \| \psi^{n}_{h} \| + \| \phi^{n}_{h} \| + \gamma^{-1}\Delta t\| \pi^{n}_{h} \|\Big\} + C\Delta t \Big\{ \frac{1}{\Delta t}\| \zeta_{t} \|^{2}_{L^{2}(t^{n},t^{n+1};L^{2}(\Omega))}
\\ + \frac{1}{\Delta t}\| \eta_{t} \|^{2}_{L^{2}(t^{n},t^{n+1};L^{2}(\Omega))} + \Delta t\| \lambda_{t} \|^{2}_{L^{2}(t^{n},t^{n+1};L^{2}(\Omega))} + h\| \nabla \zeta_{t} \|^{2}_{L^{2}(t^{n},t^{n+1};L^{2}(\Omega))} + h\| \nabla \eta_{t} \|^{2}_{L^{2}(t^{n},t^{n+1};L^{2}(\Omega))}
\\  + \| \zeta^{n} \|^{2} + \| \nabla \zeta^{n+1} \|^{2} + \| \zeta^{n+1} \| \| \nabla \zeta^{n+1}\| + \| \eta^{n} \| \| \nabla \eta^{n}\| + \| \eta^{n+1} \| \| \nabla \eta^{n+1}\| + \|\tau - I_{h}\tau\|^{2} + \|\nabla (\tau - I_{h}\tau)\|^{2} 
\\ + \Delta t \| \theta_{t} \|^{2}_{L^{2}(t^{n},t^{n+1};L^{2}(\Omega))} + \Delta t \| p_{t} \|^{2}_{L^{2}(t^{n},t^{n+1};L^{2}(\Omega))} + \Delta t \| \theta_{tt} \|^{2}_{L^{2}(t^{n},t^{n+1};L^{2}(\Omega))} + \Delta t \| u_{tt} \|^{2}_{L^{2}(t^{n},t^{n+1};L^{2}(\Omega))} 
\\ + (h + \Delta t)\| \nabla \theta_{t} \|^{2}_{L^{2}(t^{n},t^{n+1};L^{2}(\Omega))} + (h + \Delta t)\| \nabla u_{t} \|^{2}_{L^{2}(t^{n},t^{n+1};L^{2}(\Omega))}\|\Big\}.
\end{multline}
Sum from $n = 0$ to $n = N-1$, apply Lemmas \ref{l3} and \ref{l2}, take infimums over $X_{h}$, $Q_{h}$, and $\hat{W_{h}}$, and renorm. Then,
\begin{multline*}
\|\psi^{N}_{h}\|^{2} + \|\phi^{N}_{h}\|^{2} + \gamma^{-1} \Delta t \|\pi^{N}_{h}\|^{2} + \frac{1}{2}\sum^{N-1}_{n=0} \Big\{\|\psi^{n+1}_{h}-\psi^{n}_{h}\|^{2} + \|\phi^{n+1}_{h}-\phi^{n}_{h}\|^{2} + \gamma^{-1} \Delta t \|\pi^{n+1}_{h} - \pi^{n}_{h}\|^{2}\Big\}
\\  + \frac{Pr\Delta t}{4} \|\nabla \phi^{N}_{h}\|^{2} + \frac{1}{2}\vertiii{\nabla \psi_{h}}^{2}_{2,0} + \frac{Pr}{2} \vertiii{\nabla \phi_{h}}^{2}_{2,0} \leq C exp(C_{\star} t^{\ast}) \inf_{\substack{v_{h} \in X_{h}\\q_{h} \in Q_{h}\\S_{h} \in \hat{W}_{h}}} \Big\{ \| \zeta_{t} \|^{2}_{L^{2}(0,t^{\ast};L^{2}(\Omega))}
\\ + \| \eta_{t} \|^{2}_{L^{2}(0,t^{\ast};L^{2}(\Omega))} + \Delta t^{2} \| \lambda_{t} \|^{2}_{L^{2}(0,t^{\ast};L^{2}(\Omega))} + h\Delta t \| \nabla \zeta_{t} \|^{2}_{L^{2}(0,t^{\ast};L^{2}(\Omega))} + h \Delta t \| \nabla \eta_{t} \|^{2}_{L^{2}(0,t^{\ast};L^{2}(\Omega))}
\\ + \vertiii{\zeta}^{2}_{2,0} + \vertiii{\nabla \zeta}^{2}_{2,0} + \vertiii{\zeta}_{2,0} \vertiii{\nabla \zeta}_{2,0} + \vertiii{\eta}_{2,0} \vertiii{\nabla \eta}_{2,0} + t^{\ast}\big(\|\tau - I_{h}\tau\|^{2} + \|\nabla (\tau - I_{h}\tau)\|^{2}\big)
\\ + \Delta t^{2} \| \theta_{t} \|^{2}_{L^{2}(0,t^{\ast};L^{2}(\Omega))} + \Delta t^{2} \| p_{t} \|^{2}_{L^{2}(0,t^{\ast};L^{2}(\Omega))} + \Delta t^{2} \| \theta_{tt} \|^{2}_{L^{2}(0,t^{\ast};L^{2}(\Omega))} + \Delta t^{2} \| u_{tt} \|^{2}_{L^{2}(0,t^{\ast};L^{2}(\Omega))}
\\ + (h + \Delta t)\Delta t \| \nabla \theta_{t} \|^{2}_{L^{2}(0,t^{\ast};L^{2}(\Omega))} + (h + \Delta t) \Delta t \| \nabla u_{t} \|^{2}_{L^{2}(0,t^{\ast};L^{2}(\Omega))} + \|\psi^{0}_{h}\|^{2} + \|\phi^{0}_{h}\|^{2}
\\ + \gamma^{-1} \Delta t \|\pi^{0}_{h}\|^{2} + \frac{Pr\Delta t}{4}\|\nabla \phi^{0}_{h}\|^{2}\Big\}.
\end{multline*}
Assuming $\|\psi^{0}_{h}\|=\|\phi^{0}_{h}\|=\|\pi^{0}_{h}\|=\|\nabla \phi^{0}_{h}\|=0$, the result follows by the relationship $e^{n}_{T} = e^{n}_{\theta} + \tau - I_{h}\tau$, the triangle inequality, and absorbing constants.
\end{proof}
The following corollary holds for Taylor-Hood elements.
\begin{corollary}
	Suppose the assumptions of Theorem \ref{t1} hold with $k=m=2$.  Further suppose that the finite element spaces ($X_{h}$,$Q_{h}$,$W_{h}$) are given by P2-P1-P2 (Taylor-Hood), then the errors in velocity and temperature satisfy
	\begin{multline*}
\frac{1}{2}\|e^{N}_{T}\|^{2} + \|e^{N}_{u}\|^{2} + \gamma^{-1} \Delta t \|e^{N}_{p}\|^{2} + \frac{1}{2}\sum^{N-1}_{n=0} \Big\{\|e^{n+1}_{T}-e^{n}_{T}\|^{2} + \|e^{n+1}_{u}-e^{n}_{u}\|^{2} + \gamma^{-1} \Delta t \|e^{n+1}_{p} - e^{n}_{p}\|^{2}\Big\}
\\  + \frac{Pr\Delta t}{4} \|\nabla e^{N}_{u}\|^{2} + \frac{1}{4}\vertiii{\nabla e_{T}}^{2}_{2,0} + \frac{Pr}{2} \vertiii{\nabla e_{u}}^{2}_{2,0} \leq C exp(C_{\star} t^{\ast}) \Big\{ h^{6} + h^{6}\Delta t^{2} + h^{6} \Delta t + h^{5} + h^{4} + h\Delta t + \Delta t^{2} 
\\ + \|e^{0}_{T}\|^{2} + \|e^{0}_{u}\|^{2} + \gamma^{-1} \Delta t \|e^{0}_{p}\|^{2} + \frac{Pr\Delta t}{4}\|\nabla e^{0}_{u}\|^{2}\Big\}.
	\end{multline*}
\end{corollary}
Similarly, for the MINI element, the following holds.
\begin{corollary}
	Suppose the assumptions of Theorem \ref{t1} hold with $k=m=1$.  Further suppose that the finite element spaces ($X_{h}$,$Q_{h}$,$W_{h}$) are given by P1b-P1-P1b (MINI element), then the errors in velocity and temperature satisfy
	\begin{multline*}
\frac{1}{2}\|e^{N}_{T}\|^{2} + \|e^{N}_{u}\|^{2} + \gamma^{-1} \Delta t \|e^{N}_{p}\|^{2} + \frac{1}{2}\sum^{N-1}_{n=0} \Big\{\|e^{n+1}_{T}-e^{n}_{T}\|^{2} + \|e^{n+1}_{u}-e^{n}_{u}\|^{2} + \gamma^{-1} \Delta t \|e^{n+1}_{p} - e^{n}_{p}\|^{2}\Big\}
\\  + \frac{Pr\Delta t}{4} \|\nabla e^{N}_{u}\|^{2} + \frac{1}{4}\vertiii{\nabla e_{T}}^{2}_{2,0} + \frac{Pr}{2} \vertiii{\nabla e_{u}}^{2}_{2,0} \leq C exp(C_{\star} t^{\ast}) \Big\{ h^{4} + h^{4}\Delta t^{2} + h^{4} \Delta t + h^{3} + h^{2} + h\Delta t + \Delta t^{2} 
\\ + \|e^{0}_{T}\|^{2} + \|e^{0}_{u}\|^{2} + \gamma^{-1} \Delta t \|e^{0}_{p}\|^{2} + \frac{Pr\Delta t}{4}\|\nabla e^{0}_{u}\|^{2}\Big\}.
	\end{multline*}
\end{corollary}
 \begin{figure}
 	\includegraphics[width=\textwidth,height=\textheight,keepaspectratio]{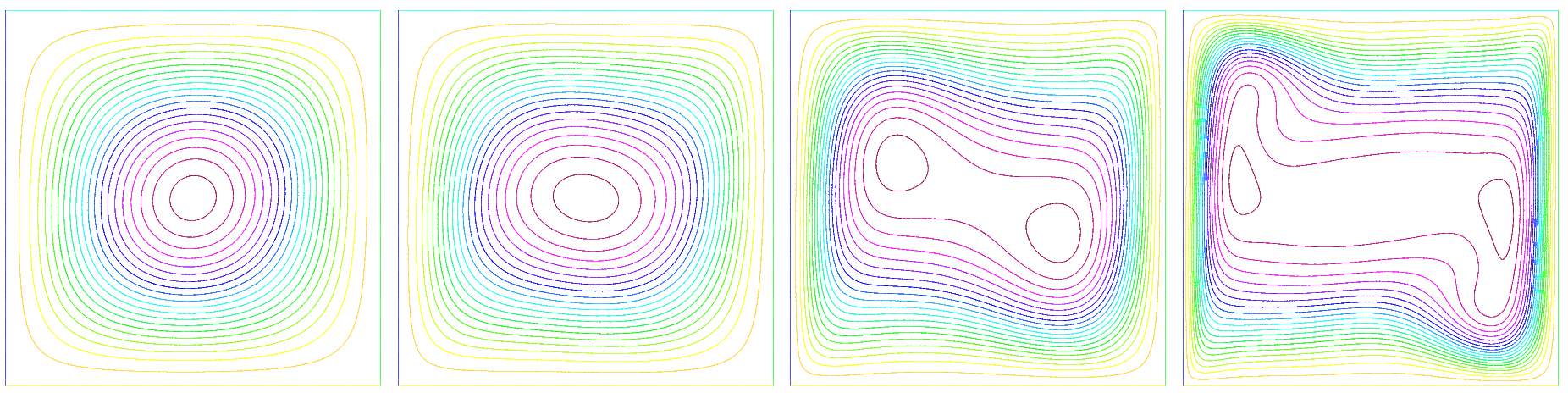}
 	\caption{Streamlines: $Ra = 10^3, 10^4, 10^5,$ and $10^6$, left to right.}\label{figure=streamlines}
 \end{figure}
 \begin{figure}
 	\includegraphics[width=\textwidth,height=\textheight,keepaspectratio]{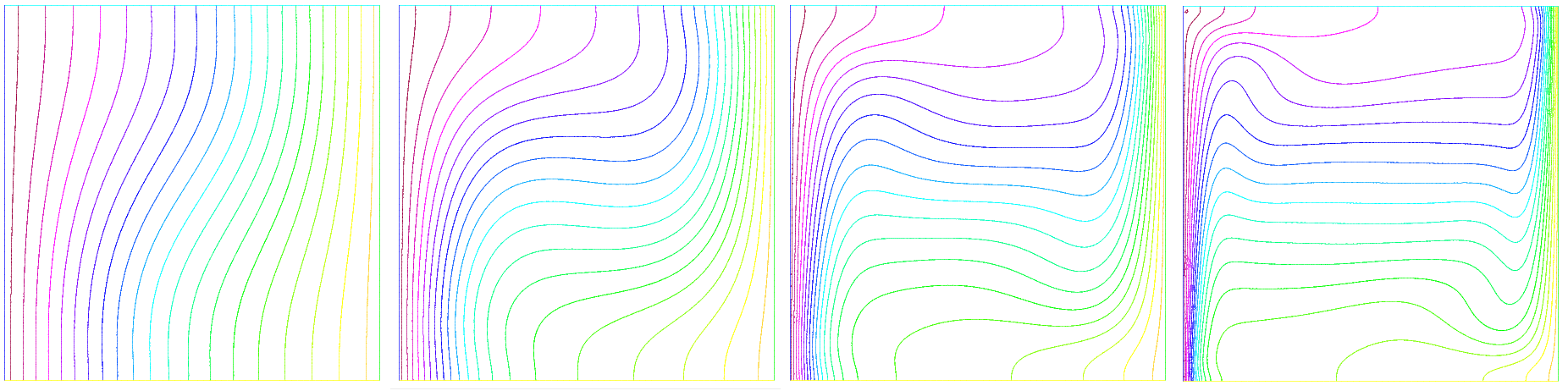}
 	\caption{Isotherms: $Ra = 10^3, 10^4, 10^5,$ and $10^6$, left to right.}\label{figure=isotherms}
 \end{figure}
\section{Numerical Experiments}\label{numerics}
In this section, we illustrate the speed, stability, and convergence of ACE described by (\ref{scheme:one:velocity}) - (\ref{scheme:one:pressure}) using Taylor-Hood (P2-P1-P2) elements to approximate the average velocity, pressure, and temperature.  The numerical experiments include the double pane window benchmark \cite{Davis}, a convergence experiment with an analytical solution devised through the method of manufactured solutions, and a predictability experiment.  In particular, ACE is shown to be 3 to 8 times faster than linearly implicit BDF1 in Section \ref{Davisproblem}.  First-order accuracy is illustrated in Section \ref{Convergencestudy}.  Lastly, in Section \ref{Predictability}, we calculate $\delta$-predictability horizons and variance to study the predictability of an unstable solution.  The software platform used for all tests is \textsc{FreeFem}$++$ \cite{Hecht}.  
\subsection{Stability condition}
Recall that ACE is stable provided condition (\ref{c1}) holds:
\begin{align*}
\frac{C_{\dagger} \Delta t}{h} \max_{1 \leq j \leq J} \|\nabla {u'}^{n}_{h}\|^{2} \leq 1.
\end{align*}
The stability constant $C_{\dagger}$ is determined via pre-computations for the double pane window benchmark; it is set to 0.35.  Condition (\ref{c1}) is checked at each timestep.  The timestep is halved and the timestep is repeated if (\ref{c1}) violated.  The timestep is never increased.  The condition is violated three times in Section \ref{Davisproblem} for $Ra = 10^6$.  

\textbf{Remark:}  Although $C_{\dagger}$ is estimated to be 1, it is set to 0.35.  This is done to reduce the timestep when $Ra = 10^6$.  At this value of $Ra$, the stopping condition is not met unless the timestep is reduced.  Instead, the solution appears to reach a false quasi-periodic solution.  This occurs for linearly implicit BDF1 and variants and may be related to the conditional Lyapunov stability of these methods \cite{Sussman}.  This is currently under investigation.

\subsection{Perturbation generation}\label{BV}
In Section \ref{Convergencestudy}, a positive and negative perturbation pair is chosen to manufacture a solution with certain properties.  The bred vector (BV) algorithm \cite{Toth} is used to generate perturbations in Sections \ref{Davisproblem} and \ref{Predictability}.  The BV algorithm simulates growth errors due to uncertainty in the initial conditions; this is neccessary and \textit{random perturbations are not sufficient} \cite{Toth}.  As a consequence, the nonlinear error growth in the ensemble average is reduced, which is witnessed in Section \ref{Predictability}.  Our experimental results are drastically different when using BVs compared to random perturbations, consistent with the above.

To begin, an initial random positive and negative perturbation pair is generated, $\pm \epsilon = \pm (\delta_{1},\delta_{2},\delta_{3},\delta_{4})$ with $\delta_{i} \in (0,0.01) \; \forall 1\leq i \leq 4$.  Denoting the control and perturbed numerical approximations $\chi^{n}_{h}$ and $\chi^{n}_{p,h}$, respectively, a bred vector $bv(\chi;\delta_{i})$ is generated via:

\underline{Algorithm: BV}\\
\indent \textbf{Step one:}  Given $\chi^{0}_{h}$ and $\delta_{i}$, put $\chi^{0}_{p,h} = \chi^{0}_{h} + \delta_{i}$.  Select time reinitialization interval $\delta t \geq \Delta t$ and let \indent $t^{k} = k \delta t$ with $0 \leq k \leq k^{\ast} \leq N$.

\indent \textbf{Step two:} Compute $\chi^{k}_{h}$ and $\chi^{k}_{p,h}$.  Calculate $bv(\chi^{k};\delta_{i}) = \frac{\delta_{i}}{\| \chi^{k}_{p,h} - \chi^{k}_{h} \|} (\chi^{k}_{p,h} - \chi^{k}_{h})$.

\indent \textbf{Step three:}  Put $\chi^{k}_{p,h} = \chi^{k}_{h} + bv(\chi^{k};\delta_{i}) $.

\indent \textbf{Step four:} Repeat from \textbf{Step two} with $k = k + 1$.

\indent \textbf{Step five:} Put $bv(\chi;\delta_{i}) = bv(\chi^{k^{\ast}};\delta_{i})$.

\noindent The bred vector pair generates a pair of initial conditions via $\chi_{\pm} = \chi^{0} + bv(\chi;\pm \delta_{i})$.  We let $k^{\ast} = 5$ and choose $\delta t = \Delta t = 0.001$ for all tests.

\begin{figure}
\includegraphics[width=\textwidth,height=\textheight,keepaspectratio]{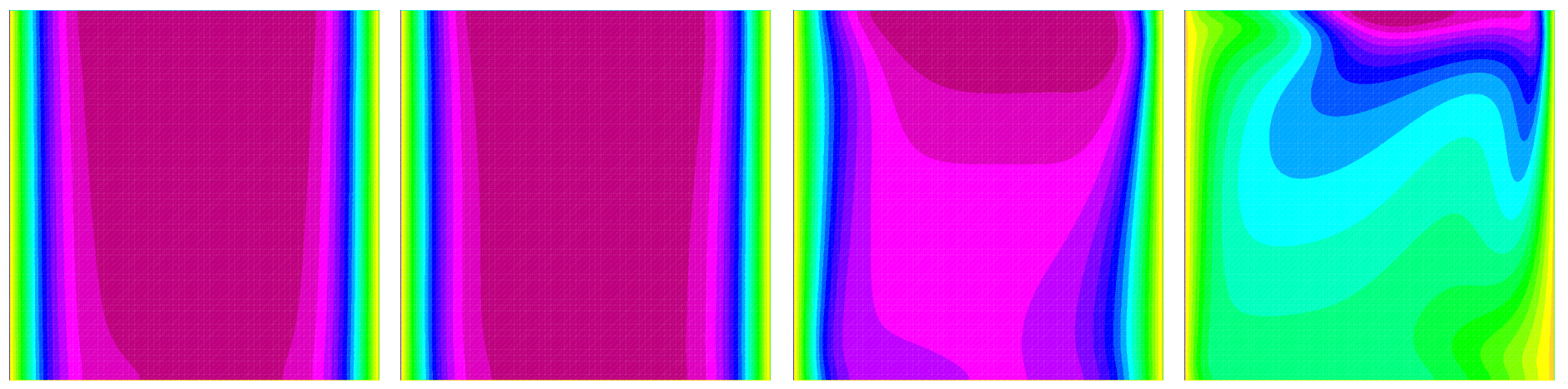}
\caption{BV ($bv(T;+\delta_{3})$): $Ra = 10^3, 10^4, 10^5,$ and $10^6$, left to right.}\label{figure=bvt}
\end{figure}
\subsection{The double pane window problem}\label{Davisproblem}
This is a classic test problem for natural convection.  The problem is the flow of air, $Pr = 0.71$, in a unit square cavity subject to no-slip boundary conditions.  The horizontal walls are adiabatic and vertical wall temperature is maintained at constant temperature \cite{Davis}; see Figure \ref{figure=domain}.  We set $\epsilon = 0.01 \Delta t$.  

We first validate our code.  We set $J=2$ and vary $Ra \in \{10^3,10^4,10^5,10^6\}$.  The finite element mesh is a division of $[0,1]^{2}$ into $64^{2}$ squares with diagonals connected with a line within each square in the same direction.  The initial timestep $\Delta t = 0.001$; it is halved three times for $Ra = 10^{6}$ to $0.000125$.  The initial conditions are generated via the BV algorithm,
\begin{align*}
u_{\pm}(x,y,0) := u(x,y,0;\omega_{1,2}) &= (u^{prev}_{1} + bv(u_{1};\pm \delta_{1}), u^{prev}_{2} + bv(u_{2};\pm \delta_{2}))^{T}, \\
T_{\pm}(x,y,0) := T(x,y,0;\omega_{1,2}) &= T^{prev} + bv(T;\pm \delta_{3}),\\
p_{\pm}(x,y,0) := p(x,y,0;\omega_{1,2}) &= p^{prev} + bv(p;\pm \delta_{4}),
\end{align*}
where the subscript \textit{prev} denotes the solution from the previous value of $Ra$; for $Ra = 10^{3}$, the previous values are all set to 1.  The BV, $bv(T;+\delta_{3})$, is presented in Figure \ref{figure=bvt}.  Forcings are identically zero for $j = 1,2$.  The stopping condition is
\begin{equation*}
\max_{0\leq n \leq N-1}\big\{\frac{\| u^{n+1}_{h} - u^{n}_{h}\|}{\| u^{n+1}_{h}\|},\frac{\|T^{n+1}_{h} - T^{n}_{h}\|}{\| T^{n+1}_{h}\|}\big\} \leq 10^{-5}.
\end{equation*}
The quantities of interest are: $\max_{y \in \Omega_{h}} {u_{1}(0.5,y,t^{\ast})}$, $\max_{x \in \Omega_{h}} {u_{2}(x,0.5,t^{\ast})}$, the local Nusselt number at vertical walls, and average Nusselt number at the hot wall.  The latter two are given by
\begin{align*}
Nu_{local} =  \pm \frac{\partial{T}}{\partial{x}},\\
Nu_{avg} =  \int^{1}_{0} Nu_{local} dy,
\end{align*}
where $\pm$ corresponds to the cold and hot walls, respectively.

Plots of $Nu_{local}$ at the hot and cold walls are presented in Figure \ref{figure=nusselt}.  Computed values of the remaining quantities are presented, alongside several of those seen in the literature, in Tables \ref{table=xhalf} - \ref{table=avgnusselt}.  Figures \ref{figure=streamlines} and \ref{figure=isotherms} present the velocity streamlines and temperature isotherms for the averages.  All results are consistent with benchmark values in the literature \cite{Davis, Manzari, Wan, Cibik, Zhang}.

The second test is a timing test comparing ACE vs. linearly implicit BDF1.  Standard GMRES is used for the velocity and temperature solves..  We set $J=1$ and vary $10^{3} \leq Ra \leq 10^{6}$.  The timestep is chosen to be $\Delta t = 0.001$ for $10^{3} \leq Ra \leq 10^{5}$ and $\Delta t = 0.0001$ for $Ra = 5\times 10^5$ and $10^6$.  The initial conditions are prescribed as in the above.  Results are presented in Figure \ref{figure=speed}.  We see that for $Ra = 10^3$, both algorithms have increased runtimes relative to all other cases.  This is due to the relatively poor choice of initial condition.  Moreover, linearly implicit BDF1 suffers from increased runtime with increasing $Ra$.  However, ACE runtimes remain relatively constant.  Overall, ACE is 3 to 8 times faster for this test problem.

 \begin{figure}
 	\centering
 	\includegraphics[width=5.5in,height=\textheight, keepaspectratio]{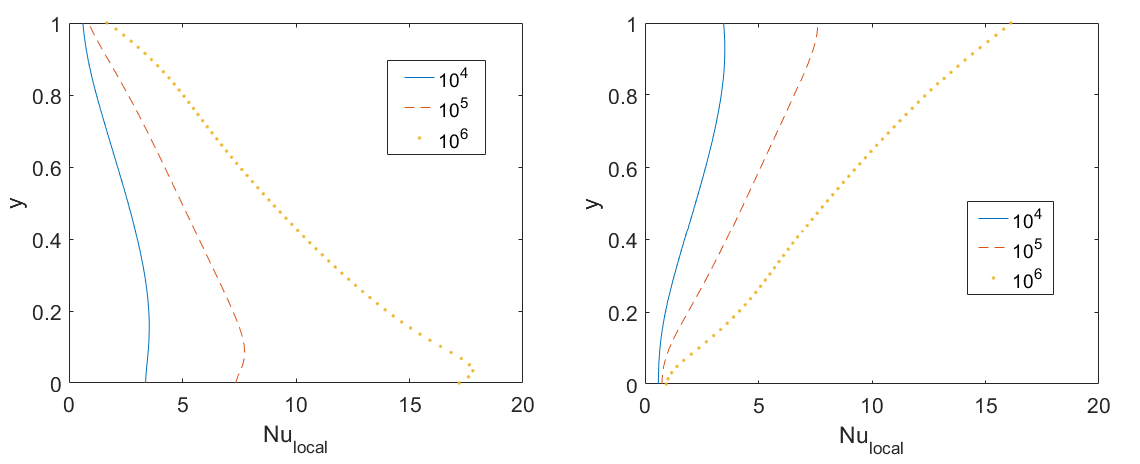}
 	\caption{Variation of the local Nusselt number at the hot (left) and cold walls (right).}\label{figure=nusselt}
 \end{figure}
 \vspace{5mm}
 \begin{adjustbox}{max width=\textwidth}
 	\begin{tabular}{ c  c  c  c  c  c  c  c }
 		\hline			
 		Ra & Present study & Ref. \cite{Davis} & Ref. \cite{Manzari} & Ref. \cite{Wan} & Ref. \cite{Cibik} & Ref. \cite{Zhang} \\
 		\hline
 		$10^{4}$ & 16.16 (64$\times$64) & 16.18 (41$\times$41) & 16.10 (71$\times$71) & 16.10 (101$\times$101) & 15.90 (11$\times$11) & 16.18 (64$\times$64)\\
 		$10^{5}$ & 34.65 (64$\times$64) & 34.81 (81$\times$81) & 34 (71$\times$71) & 34 (101$\times$101) & 33.51 (21$\times$21) & 34.74 (64$\times$64) \\
 		$10^{6}$ & 65.48 (64$\times$64) & 65.33 (81$\times$81) & 65.40 (71$\times$71) & 65.40 (101$\times$101) & 65.52 (32$\times$32) & 64.81 (64$\times$64)\\
 		\hline  
 	\end{tabular}
 \end{adjustbox}
 \captionof{table}{Comparison: maximum horizontal velocity at x = 0.5 \& mesh size, double pane window problem.}\label{table=xhalf}
 \begin{adjustbox}{max width=\textwidth}
 	\begin{tabular}{ c  c  c  c  c  c  c  c }
 		\hline			
 		Ra & Present study & Ref. \cite{Davis} & Ref. \cite{Manzari} & Ref. \cite{Wan} & Ref. \cite{Cibik} & Ref. \cite{Zhang} \\
 		\hline
 		$10^{4}$ & 19.65 (64$\times$64) & 19.51 (41$\times$41) & 19.90 (71$\times$71) & 19.79 (101$\times$101) & 19.91 (11$\times$11) & 19.62 (64$\times$64)\\
 		$10^{5}$ & 68.88 (64$\times$64) & 68.22 (81$\times$81) & 70 (71$\times$71) & 70.63 (101$\times$101) & 70.60 (21$\times$21) & 68.48 (64$\times$64) \\
 		$10^{6}$ & 218.63 (64$\times$64) & 216.75 (81$\times$81) & 228 (71$\times$71) & 227.11 (101$\times$101) & 228.12 (32$\times$32) & 220.44 (64$\times$64)\\
 		\hline  
 	\end{tabular}
 \end{adjustbox}
 \captionof{table}{Comparison: maximum vertical velocity at y = 0.5 \& mesh size, double pane window problem.}\label{table=yhalf}
 \begin{adjustbox}{max width=\textwidth}
 	\begin{tabular}{ c  c  c  c  c  c  c  c }
 		\hline			
 		Ra & Present study & Ref. \cite{Davis} & Ref. \cite{Manzari} & Ref. \cite{Wan} & Ref. \cite{Cibik} & Ref. \cite{Zhang} \\
 		\hline
 		$10^{4}$ & 2.24 (64$\times$64) & 2.24 (41$\times$41) & 2.08 (71$\times$71) & 2.25 (101$\times$101) & 2.15 (11$\times$11) & 2.25 (64$\times$64)\\
 		$10^{5}$ & 4.50 (64$\times$64) & 4.52 (81$\times$81) & 4.30 (71$\times$71) & 4.59 (101$\times$101) & 4.35 (21$\times$21) & 4.53 (64$\times$64) \\
 		$10^{6}$ & 8.77 (64$\times$64) & 8.92 (81$\times$81) & 8.74 (71$\times$71) & 8.97 (101$\times$101) & 8.83 (32$\times$32) & 8.87 (64$\times$64)\\
 		\hline  
 	\end{tabular}
 \end{adjustbox}
 \captionof{table}{Comparison: average Nusselt number on vertical boundary x = 0 \& mesh size, double pane window problem.}\label{table=avgnusselt}
 \begin{figure}
 	\centering
\includegraphics[width=2.85in,height=2.85in, keepaspectratio]{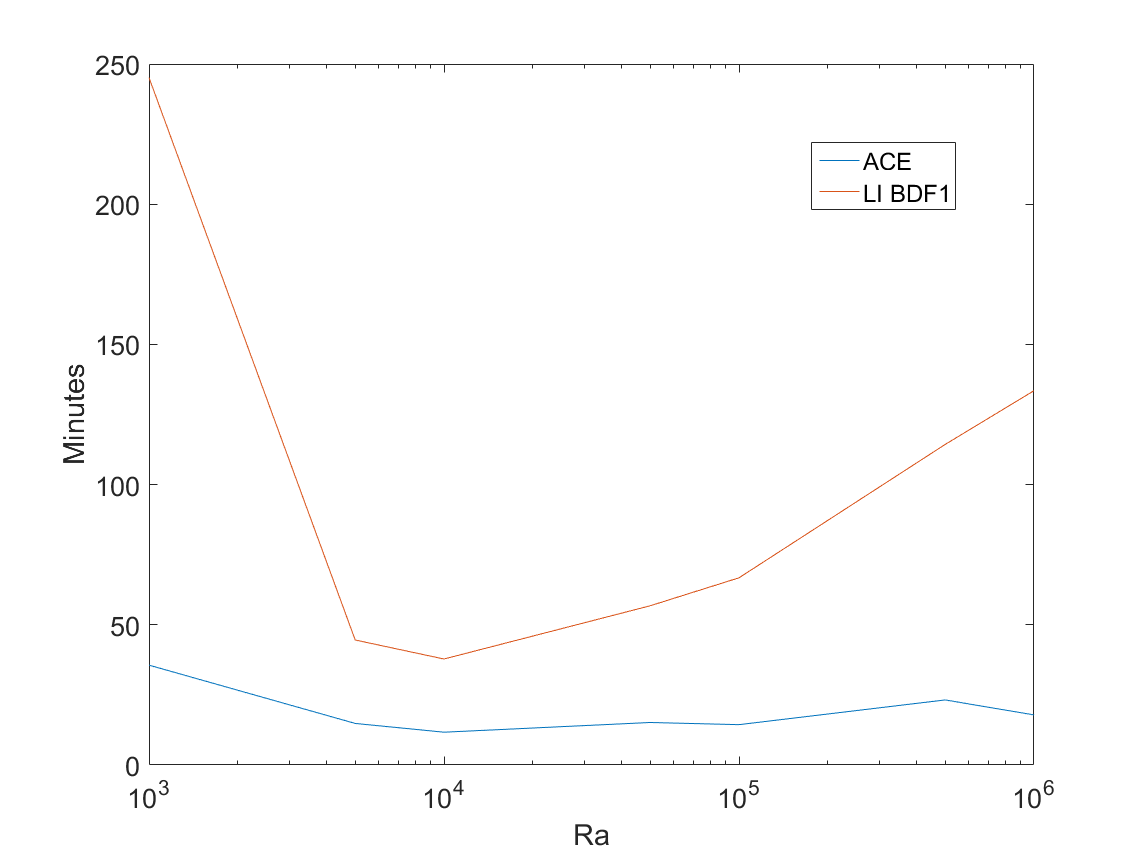} 
 	\caption{Time: ACE vs. standard, coupled linearly implicit BDF1, double pane window benchmark.}\label{figure=speed}
 \end{figure}
\subsection{Numerical convergence study}\label{Convergencestudy}
We now illustrate convergence rates for ACE (\ref{scheme:one:velocity}) - (\ref{scheme:one:temperature}).  The domain and parameters are $\Omega = [0,1]^{2}$, $Pr = 1.0$, and $Ra = 100$.  The unperturbed solution is given by
\begin{align}\label{manufactured1}
u(x,y,t) &= A(t)(x^2(x-1)^2y(y-1)(2y-1), -x(x-1)(2x-1)y^2(y-1)^2)^T, \\
T(x,y,t) &= u_{1}(x,y,t) + u_{2}(x,y,t), \\
p(x,y,t) &= A(t)(2x-1)(2y-1),\label{manufactured3}
\end{align}
with $A(t) = 10\cos{(t)}$.  Perturbed solutions are given by
\begin{align*}
u(x,y,t;\omega_{1,2}) = (1 + \delta_{1,2})u(x,y,t), \\
T(x,y,t;\omega_{1,2}) = (1 + \delta_{1,2})T(x,y,t), \\
p(x,y,t;\omega_{1,2}) = (1 + \delta_{1,2})p(x,y,t), 
\end{align*} 
where $\delta_{1} = 1e-3 = -\delta_{2}$, and satisfy the following relations
\begin{align*}
< u > = 0.5\big(u(x,y,t;\omega_{1}) + u(x,y,t;\omega_{2}) \big) = u(x,y,t), \\
< T > = 0.5\big(T(x,y,t;\omega_{1}) + T(x,y,t;\omega_{2}) \big) = T(x,y,t), \\
< p > = 0.5\big(p(x,y,t;\omega_{1}) + p(x,y,t;\omega_{2}) \big) = p(x,y,t).
\end{align*}
Forcings and boundary conditions are adjusted appropriately.  The mesh is constructed via Delaunay triangulation generated from $m$ points on each side of the boundary.  We calculate errors in the approximations of the average velocity, temperature, and pressure with the $L^{\infty}(0,t^{\ast};L^{2}(\Omega))$ norm.  Rates are calculated from the errors at two successive $\Delta t_{1,2}$ via
\begin{align*}
\frac{\log_{2}(e_{\chi}(\Delta t_{1})/e_{\chi}(\Delta t_{2}))}{\log_{2}(\Delta t_{1}/\Delta t_{2})},
\end{align*}
respectively, with $\chi = u, T, p$.  We set $\Delta{t} = \frac{1}{10m}$ and vary $m$ between 8, 16, 24, 32, and 40.  Results are presented in Table \ref{table=conv}.  First-order convergence is observed for each solution variable.  The results for velocity and temperature are predicted by our theory; however, pressure is a half-power better than predicted.

 \vspace{5mm}
 \begin{center}
 \begin{adjustbox}{max width=\textwidth}
 \begin{tabular}{ c  c  c  c  c  c  c }
 	\hline			
 	$m$ & $\vertiii{ <u_{h}>- u }_{\infty,0}$ & Rate & $\vertiii{ <T_{h}> - T }_{\infty,0}$ & Rate & $\vertiii{ <p_{h}> - p }_{\infty,0}$ & Rate \\
 	\hline
 	8 & 0.0083577 & - & 1.20E-04 & - & 0.15973 & -\\
 	16 & 0.0042676 & 0.97 & 1.51E-05 & 2.99 & 0.073252 & 1.12 \\
 	24 & 0.0028632 & 0.98 & 4.67E-06 & 2.89 & 0.047944 & 1.04 \\
 	32 & 0.0021495 & 1.00 & 2.40E-06 & 2.31 & 0.035660 & 1.03 \\
 	40 & 0.0017263& 0.98 & 1.68E-06 & 1.62 & 0.028505 & 1.00 \\
 	\hline  
 \end{tabular}
 \end{adjustbox}
 \captionof{table}{Errors and rates for average velocity, temperature, and pressure in corresponding norms.}\label{table=conv}
 \end{center}
\subsection{Exploration of predictability}\label{Predictability}
We now illustrate the usefulness of ensembles.  The domain $\Omega$ and $Pr$ are the same as in Section \ref{Convergencestudy}.  We also consider the manufactured solution (\ref{manufactured1}) - (\ref{manufactured3}) with $A(t) = 10(1+0.1t)\cos{(t)}$.  We set $\epsilon = \Delta t$ and vary $Ra \in \{10^2,10^3,10^4\}$.  Forcing and boundary conditions are adjusted appropriately.  Instead of specifying the perturbations on the initial conditions, we utilize the BV algorithm as in Section \ref{Davisproblem}.  The initial timestep is $\Delta t = 0.001$.  The final time $t^{\ast} = 0.1$.  We utilize the following definitions of energy, variance, average effective Lyapunov exponent \cite{Boffetta}, and $\delta$-predictability horizon \cite{Boffetta}.
\begin{definition} The energy is given by
	\begin{align*}
	Energy := \| T \|+ \frac{1}{2}\| u \|^{2}.
	\end{align*}
	The variance of $\chi$ is
	\begin{align*}
	V(\chi) := < \| \chi \|^{2} > - \| <\chi> \|^{2} = < \| \chi'\|^{2} >.
	\end{align*} 
	The relative energy fluctuation is
	\begin{align*}
	r(t) := \frac{\| \chi_{+} - \chi_{-} \|^{2}}{\| \chi_{+} \| \| \chi_{-} \|},
	\end{align*}
	and the \textit{average effective Lyapunov exponent} over $0 < \tau \leq t^{\ast}$ is
	\begin{align*}
	\gamma_{\tau}(t) := \frac{1}{2 \tau}\log\big(\frac{r(t + \tau)}{r(t)}\big),
	\end{align*}
	with $0 < t + \tau \leq t^{\ast}$.  The $\delta$-\textit{predictability horizon} is
	\begin{align*}
	t_{p} := \frac{1}{\gamma_{t^{\ast}}(0)}\log\Big(\frac{\delta}{\| (\chi_{+} - \chi_{-}) (0) \|}\Big).
	\end{align*}
\end{definition}
Figure \ref{figure=energy} presents the energy of the approximate solutions with varying $Ra$.  Variance is presented in Figure \ref{figure=variance}.  In all cases, the ensemble average and unperturbed solution are in close agreement.  Moreover, the perturbed solutions deviate significantly from the unperturbed solution with increasing $Ra$.   Figure \ref{figure=variance}, in particular, indicates that small perturbations in the initial conditions yield unreliable velocity and pressure distributions.  On the other hand, the temperature distribution is reliable throughout the simulation.\\

The average effective Lyapunov exponents are presented in Figure \ref{figure=lyapunov} and $\delta$-predictability horizons are tabulated in Table \ref{table=horizon} for $\delta = e\| (\chi_{+} - \chi_{-}) (0) \|$.  We see that $\gamma_{t^{\ast}}(0)$ is positive for each solution variable, indicating finite time flow predictability.  Moreover, it becomes increasingly larger (reduced predictability) with increasing $Ra$.  For velocity and temperature, $\gamma_{t^{\ast}}(t)$ remains positive, becoming increasingly larger with time; in other words, increasingly unpredictable.  For the pressure, however,  $\gamma_{t^{\ast}}(t)$ becomes and stays negative, indicating increasing predictability.  These results seem to be, in part, inconsistent with the variance plots, Figure \ref{figure=variance}.  It is unclear how to interpret this inconsistency.
 \vspace{2mm}
 \begin{center}
 \begin{adjustbox}{max width=\textwidth}
 \begin{tabular}{ c  c  c  c }
 	\hline			
 	 $Ra$ & u & T & p \\
 	\hline
 	$10^2$ & 0.0214 & 0.0224 & 0.0703 \\
 	$10^3$ & 0.0152 & 0.0223 & 0.0242 \\
 	$10^4$ & 0.0096 & 0.0214 & 0.0134 \\ 
 	\hline  
 \end{tabular}
 \end{adjustbox}
 \captionof{table}{$\delta$-predictability horizons for varying $Ra$.}\label{table=horizon}
 \end{center}
 
 \begin{figure}
	\centering
	\includegraphics[width=\textwidth,height=\textheight,keepaspectratio]{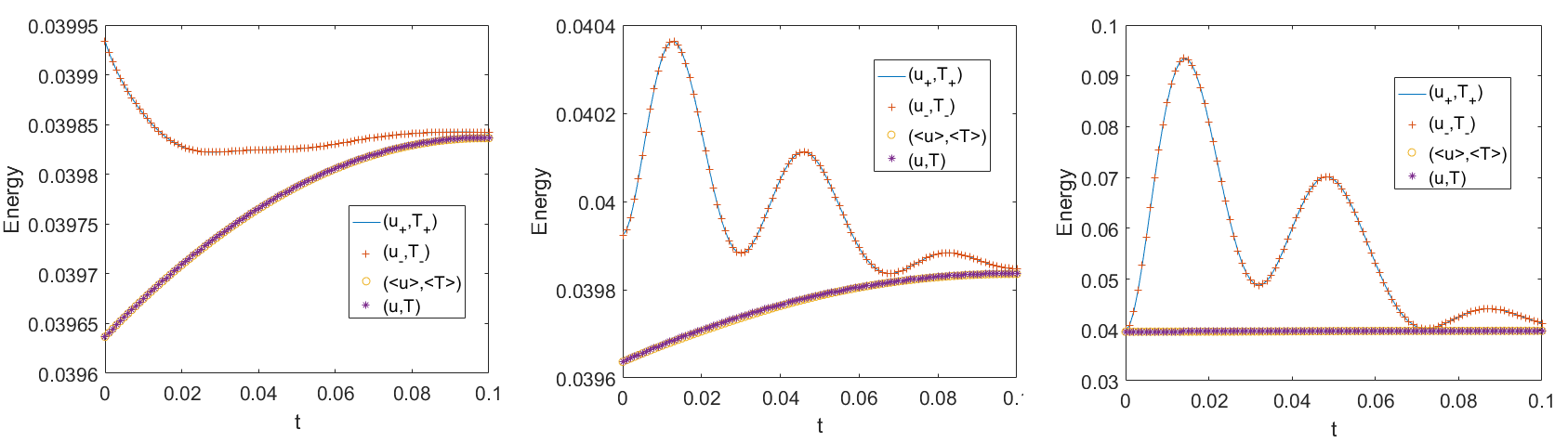}
	\caption{Energy in the system for varying $Ra = 10^2$(left), $10^3$(center), and $10^4$(right).}\label{figure=energy}
\end{figure}
\begin{figure}
	\centering
	\includegraphics[width=\textwidth,height=\textheight,keepaspectratio]{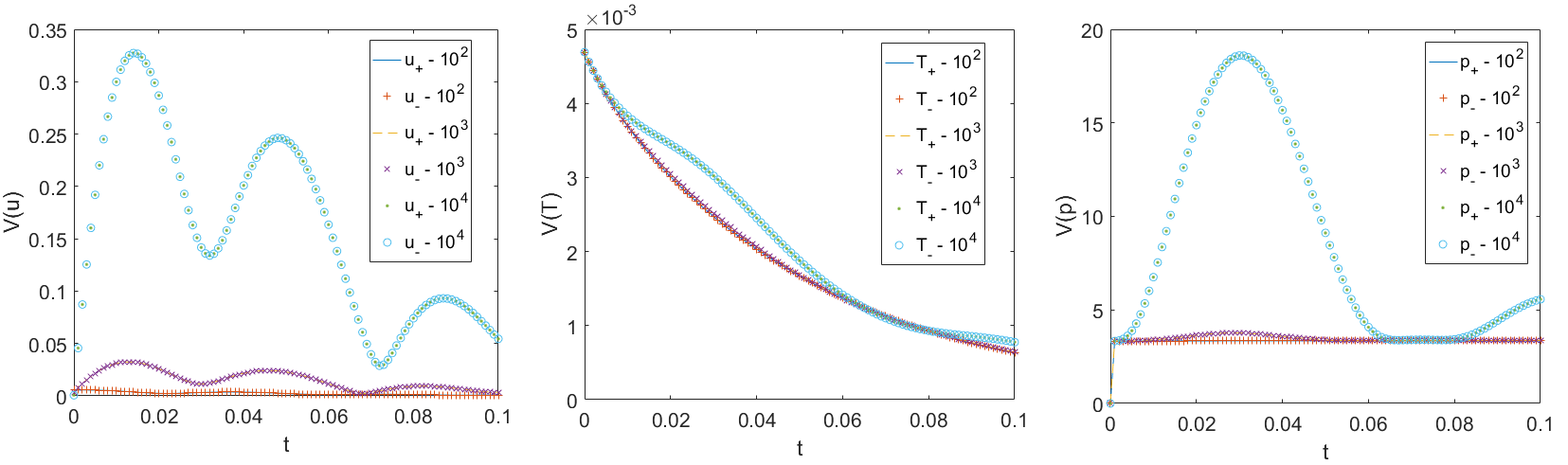}
	\caption{Variance of velocity (left), temperature (center), and pressure (right) with varying $Ra$.}\label{figure=variance}
\end{figure}
\section{Conclusion}
An efficient artificial compressibility ensemble (ACE) algorithm was presented.  Complexity and computation time are reduced compared to similar algorithms in the literature.  This is achieved via a particular IMEX splitting of the convective terms and full velocity, pressure, and temperature decoupling, utilizing artificial compressibility.  Consequently, two linear systems must be solved for multiple right-hand sides and an algebraic update at each timestep are required.  Nonlinear, energy, stability and first-order convergence were proven.  Numerical experiments were performed to illustrate proposed properties.

 \begin{figure}
 	\includegraphics[width=\textwidth,height=\textheight,keepaspectratio]{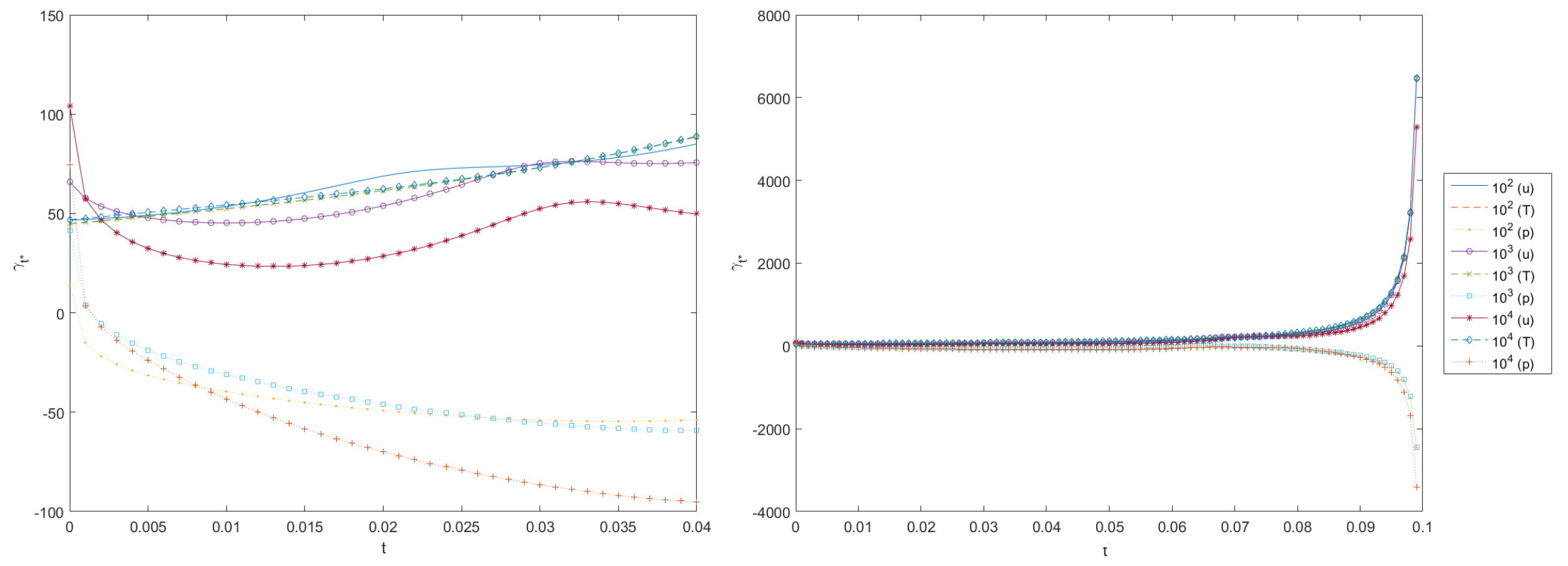}
 	\caption{$\gamma_{t^{\ast}}$ vs $t$ (right) zoomed in (left).}\label{figure=lyapunov}
 \end{figure}

\end{document}